\documentclass[a4paper,reqno]{amsart} %
\renewcommand{\a}{\alpha}
\renewcommand{\b}{\beta}
\newcommand{\V}{\mathcal{V}}          
\renewcommand{\H}{\mathbf H}          
\newcommand{\N}{\mathbb N}            

\newcommand{\SSS}{\mathbf{s}}

\DeclareMathOperator{\ad}{ad}
\DeclareMathOperator{\ch}{char}

\DeclareMathOperator{\St}{St}              
\DeclareMathOperator{\gr}{gr}

\newtheorem{Th}{Theorem}[section]
\newtheorem{Pro}[Th]{Proposition}
\newtheorem{Lemma}[Th]{Lemma}
\newtheorem{Corr}[Th]{Corollary}

\begin{document}
\title{Lie structure of truncated symmetric Poisson algebras}
\author{Ilana Z. Monteiro Alves}
\address{Department of Mathematics, Federal University of Amazonas, Humait\'a,  Amazonas, Brazil}
\email{calmariezuila@gmail.com}
\author{Victor Petrogradsky}
\address{Department of Mathematics, University of Brasilia, 70910-900 Brasilia DF, Brazil}
\email{petrogradsky@rambler.ru}
\thanks{The author was partially supported by grants CNPq, FEMAT\\\bf \today}
\begin{abstract}
The paper naturally continues series of works on identical relations of group rings, enveloping algebras,
and other related algebraic structures.
Let $L$ be a Lie algebra over a field of characteristic $p>0$.
Consider its symmetric algebra $S(L)=\oplus_{n=0}^\infty U_n/U_{n-1}$,
which is isomorphic to a polynomial ring.
It also has a structure of a Poisson algebra,
where the Lie product is traditionally denoted by $\{\ ,\ \}$.
This bracket naturally induces the structure of a Poisson algebra on the ring
$\SSS(L)=S(L)/(x^p\,|\, x\in L)$, which we call a truncated symmetric Poisson algebra.
We study Lie identical relations of $\SSS(L)$.
Namely, we determine necessary and sufficient conditions for $L$ under which $\SSS(L)$
is Lie nilpotent, strongly Lie nilpotent, solvable and strongly solvable,
where we assume that $p>2$ to specify the solvability.
We compute the strong Lie nilpotency class of $\SSS(L)$.
Also, we prove that the Lie nilpotency class coincides with the strong Lie nilpotency class in case $p>3$.

Shestakov proved that the symmetric algebra $S(L)$ of an arbitrary Lie algebra $L$ satisfies the identity
$\{x,\{y,z\}\}\equiv 0$ if, and only if, $L$ is abelian.
We extend this result for the (strong) Lie nilpotency and the (strong) solvability of $S(L)$.
We show that the solvability of $\SSS(L)$ and $S(L)$ in case $\ch K=2$ is different to other characteristics,
namely, we construct examples of such algebras which are solvable but not strongly solvable.

We use delta-sets for Lie algebras and the theory of identical relations of Poisson algebras.
Also, we study filtrations in Poisson algebras and prove results on
products of terms of the lower central series for Poisson algebras.
\end{abstract}

\maketitle
\section{Introduction}
The theory of associative PI-algebras, i.e. algebras satisfying nontrivial polynomial identities,
is  a classical area of the modern algebra~\cite{Drensky}.
This is an important instrument to study structure and properties of associative algebras.
Now, there is an established theory of identical relations in Lie algebras~\cite{Ba}.
It has many applications to group theory such as the solution of the Restricted Burnside Problem.
Also, identical relations were applied to study another algebraic structures.

The  first starting point for our research is the result of Passman on existence
of identical relations in group rings~\cite{Pas72} (Theorem~\ref{TKG}).
This paper caused an intensive research on different types of identical relations in group rings,
such as Lie nilpotence, solvability, non-matrix identical relations,
classes of Lie nilpotence, solvability lengths, etc.
There are at least 50 papers published in this area.

Second, Latyshev~\cite{Lat63} and Bahturin~\cite{Ba74}
started to study identical relations in universal enveloping algebras of Lie algebras.
Passman~\cite{Pas90} and Petrogradsky~\cite{Pe91} specified
existence of identical relations in restricted enveloping algebras (Theorem~\ref{TuL}).
There are many papers in this area studying different types of identical relations,
such as Lie nilpotence, solvability, non-matrix identical relations, classes of Lie nilpotence, solvability lengths, etc.
So, Riley and Shalev determined necessary and sufficient conditions for restricted Lie algebras under which
the restricted enveloping algebra is Lie nilpotent or solvable~\cite{RiSh93}.
The research was further extended to new objects, such as Lie superalgebras, color Lie superalgebras, smash products.
These problems were studied in numerous
papers by Bahturin, Bergen, Kochetov, Petrogradsky, Riley, Shalev, Siciliano, Spinelly, Usefi, et.al.

Poisson algebras appeared in works of Berezin~\cite{Ber67} and Vergne~\cite{Ver69}.
Free Poisson algebras were introduced by Shestakov~\cite{Shestakov93}.
A basic theory of identical relations for Poisson algebras  was developed by Farkas~\cite{Farkas98,Farkas99}.
Identical relations of symmetric Poisson algebras
of Lie (super)algebras were studied by Kostant~\cite{Kos81}, Shestakov~\cite{Shestakov93}, and Farkas~\cite{Farkas99}.
The third starting point for our research is the result of Giambruno and Petrogradsky~\cite{GiPe06}
on existence of non-trivial multilinear Poisson identical relations in truncated symmetric algebras
of restricted Lie algebras (Theorem~\ref{Tmain}).
\medskip

In Section~\ref{SPoisson} we supply basic facts on Poisson algebras and their identities.
In Section~\ref{Sidentities} we collect principal results on identical relations of group rings, enveloping algebras and
Poisson symmetric algebras that motivated our research.
Our main results are formulated in Section~\ref{Smain} (Theorem~\ref{Tnilp} and Theorem~\ref{Tsolv}).
In Section~\ref{Sfiltr} we study filtrations in (truncated) symmetric Poisson algebras.
In Section~\ref{Scomm} we prove technical results on products of terms of the lower central series of Poisson algebras,
which are analogues of respective results for associative algebras.
\medskip

By $K$ denote the ground field, as a rule of positive characteristic $p$.
By $\langle S\rangle$ or $\langle S\rangle_K$ denote the linear span of a subset $S$ in a $K$-vector space.
Let $L$ be a Lie algebra.
The Lie brackets are left-normed: $[x_1,\ldots,x_n]=[[x_1\dots,x_{n-1}],x_n]$, $n\ge 1$.
Denote $\ad x:L\to L$ by $\ad x(y)=[x,y]$, $x,y\in L$.
One defines the {\it lower central series}: $\gamma_1(L)=L$, $\gamma_{n+1}(L)=[\gamma_{n}(L), L]$, $n\ge 1$.
Also, $L^2=[L,L]=\gamma_2(L)$ is the {\it commutator subalgebra}.
By $U(L)$ denote the universal enveloping algebra and $S(L)=\mathop{\oplus}\limits_{n=0}^\infty U_n/U_{n-1}$
the related {\it symmetric algebra}~\cite{Ba,BMPZ,Dixmier}.
On restricted Lie algebras and restricted enveloping algebras see~\cite{Ba,JacLie}.
Let us note that all our Lie algebras over a field of positive characteristic need not be restricted.

\section{Poisson algebras and their identities}
\label{SPoisson}
\subsection{Poisson algebras}
Poisson algebras naturally appear in different areas of algebra, topology and physics.
Poisson algebras probably first were introduced in 1976 by Berezin~\cite{Ber67}, see also Vergne~\cite{Ver69} (1969).
Poisson algebras are used to study universal enveloping algebras of
finite dimensional Lie algebras in zero characteristic~\cite{Kos81,Ooms12}.
In particular, abelian subalgebras in symmetric Poisson algebras are used to study
commutative subalgebras in universal enveloping algebras of finite-dimensional semi-simple Lie algebras
in zero characteristic~\cite{Tar00,Vin90}.
Applying Poisson algebras, Shestakov and Umirbaev managed to solve a long-standing problem:
they proved that the Nagata automorphism of the polynomial ring in three variables $\mathbb{C}[x,y,z]$ is wild~\cite{SheUmi04}.
Related algebraic properties of free Poisson algebras were studied by
Makar-Limanov, Shestakov and Umirbaev~\cite{MakShe12,MakUmi11}.

The free Poisson algebras were defined by Shestakov~\cite{Shestakov93}.
A basic theory of identical relations for Poisson algebras  was developed by Farkas~\cite{Farkas98,Farkas99}.
See further developments on the theory of identical relations of Poisson algebras,
in particular, theory of so called  codimension growth in characteristic zero
by Mishchenko, Petrogradsky, and Regev~\cite{MiPeRe}, and Ratseev~\cite{Rats14}.
\medskip

Throughout this paper, as a rule,  $K$ denotes a field of characteristic $p>0$.
Recall that a vector space $A$ is a {\it Poisson algebra} provided that, beside the addition,
$A$ has two $K$-bilinear operations which are related by the Leibnitz rule. More precisely,
\begin{itemize}
\item
A is a commutative associative algebra with unit, denote the multiplication by $a\cdot b$ (or $ab$), where $a, b\in A$;
\item
$A$ is a Lie algebra which product is traditionally denoted by the Poisson bracket $\{a, b\}$, where $a, b\in A$;
\item these two operations are related by the Leibnitz rule
\begin{equation*}
\{a\cdot b, c\}=a\cdot\{b, c\}+b\cdot\{a, c\},\qquad  a, b, c \in A.
\end{equation*}
\end{itemize}
\subsection{Examples of Poisson algebras}
Typical examples are as follows.
\medskip

\noindent {\bf Example 1.}\  Consider the polynomial ring
$\H_{2m}=K[X_1,\dots,X_m,Y_1,\dots,Y_m]$.
Set $\{X_i,Y_j\}=\delta_{i,j}$ and extend this bracket by the Leibnitz rule. We obtain the Poisson bracket:
$$\{f,g\}=\sum_{i=1}^m
\bigg(\frac{\partial f}{\partial X_i}\;\frac{\partial g}{\partial Y_i}-
\frac{\partial f}{\partial Y_i}\;\frac{\partial g}{\partial X_i}\bigg),
\qquad f,g\in \H_{2m}.
$$
The commutative product is the natural multiplication.
We obtain the {\it Hamiltonian} Poisson algebra~$\H_{2m}$.
\medskip

\noindent {\bf Example 2.}\
Let $L$ be a Lie algebra over an arbitrary field $K$,
$\{U_n\mid n\ge 0\}$ the natural filtration of its universal enveloping algebra $U(L)$.
Consider the {\it symmetric algebra}
$S(L)=\gr U(L)=\mathop{\oplus}\limits_{n=0}^\infty U_{n}/U_{n+1}$ (see~\cite{Dixmier}).
Recall that $S(L)$ is identified with the polynomial ring $K[v_i\,|\, i\in I]$, where $\{v_i\,|\, i\in I\}$ is a $K$-basis of $L$.
Define the Poisson bracket as follows.
Set $\{v_i,v_j\}=[v_i,v_j]$ for all $i,j\in I$,
and extend to the whole of $S(L)$ by linearity and using the Leibnitz rule. For example,
$$\{v_i\cdot v_j,v_k\}=v_i\cdot\{v_j,v_k\}+v_j\cdot\{v_i,v_k\},\qquad i,j,k\in I.$$
Thus, $S(L)$ has a structure of a Poisson algebra, called the {\it symmetric algebra} of $L$.
\medskip

\noindent {\bf Example 3.}\ Let $L$ be a Lie algebra with a $K$-basis $\{v_i\,|\, i\in I\}$, where $\ch K=p>0$.
Consider a factor algebra of the symmetric (Poisson) algebra
$$\mathbf{s}(L)=S(L)/(v^p\, |\, v\in L)\cong K[v_i\mid i\in I]/(v_i^p \,|\, i\in I),$$
we get an algebra of truncated polynomials. Observe that
$$\{v^p,u\}=pv^{p-1}\{v,u\}=0,\qquad v\in L,\ u\in \SSS(L).$$
So, the Poisson bracket on $S(L)$ yields a Poisson bracket on $\SSS(L)$.
Thus, $\SSS(L)$ is a Poisson algebra, we call it {\it truncated symmetric algebra}.
Remark that $L$ need not be a restricted Lie algebra.

\medskip
\textbf{Example 4.}
Let $K$ be a field of positive characteristic $p$.
We introduce the {\it truncated Hamiltonian} Poisson  algebra as
$$\mathbf{h}_{2m}(K)=K[X_1,\dots,X_m,Y_1,\dots,Y_m]/(X_i^p,Y_i^p\mid  i=1,\dots,m),$$
where we define the bracket as in Example~1 using the observation of Example~3.
\medskip

The Hamiltonian algebras $\mathbf{h}_{2m}(K)$ and $\mathbf{H}_{2m}(K)$ in the class of Poisson algebras
play a role  similar to that of the matrix algebras $\mathrm{M}_n(K)$ for associative algebras.

\subsection{Poisson identities}
The objective of this subsection is to supply basic facts
on polynomial identities of Poisson algebras.

Consider the free Lie algebra $L=L(X)$ generated by a set $X$ and its symmetric algebra $F(X)=S(L(X))$.
Then, $F(X)$ is a {\it free Poisson algebra} in $X$, as was shown by Shestakov~\cite{Shestakov93}.
For example, let $L=L(x,y)$ be the free Lie algebra of rank 2. Consider its Hall basis~\cite{Ba}
$$L=\langle x, y, [y,x], [[y,x],x], [[y,x],y], [[[y,x],x],x],\ldots\rangle_K. $$
We obtain the free Poisson algebra $F(x,y)=S(L)$ of rank 2,
which has a canonical basis as follows:
$$
F(x,y)=\big\langle x^{n_1} y^{n_2} \{y,x\}^{n_3} \{\{y,x\},x\}^{n_4} \{\{y,x\},y\}^{n_5} \{\{\{y,x\},x\},x\}^{n_6}\cdots\,\Big|\, n_i\ge 0
\big\rangle_K,
$$
where only finitely many $n_i$, $i\ge 1$, are non-zero in the monomials above.
\medskip

A definition of  a {\it Poisson PI-algebra} is standard,
identities being elements of the free Poisson algebra  $F(X)$ of countable rank.
Assume that basic facts on identical relations of linear algebras are known to the reader,
(see, e.g.,~\cite{Ba,Drensky}).
Farkas introduced so called {\it customary identities}~\cite{Farkas98}:
$$
\sum_{\substack{\sigma\in S_{2n}\\ \sigma(2k-1)<\sigma(2k),\ k=1,\dots,n\\\sigma(1)<\sigma(3)<\cdots<\sigma(2n-1)}}
\!\!\!\!\!\!\!\!\!\!\!\!\!
\mu_{\sigma} \{x_{\sigma (1)},x_{\sigma (2)}\}\cdots
\{x_{\sigma (2n-1)},x_{\sigma (2n)}\}\equiv 0,\quad \mu_\sigma\in K.
$$
where $\mu_e=1$, for the identity permutation.
Denote by $T_{2n}$ the set of permutations $\tau\in S_{2n}$ appearing in the some above.
The importance of customary identities is explained by the following fact.

\begin{Th}[\cite{Farkas98}]\label{TFarkas0}
Suppose that $\V$ is a nontrivial variety of Poisson algebras over
a field $K$ of characteristic zero.
Then $\V$ satisfies a nontrivial customary identity.
\end{Th}

Let us show the idea of the proof.
Let a Poisson algebra $R$ satisfy the identity $f(X,Y,Z)=\{\{X,Y\},Z\}\equiv 0$.
Then, $R$ also satisfies the identity:
\begin{align*}
0&\equiv f(X_1X_2,Y,Z)-X_1f(X_2,Y,Z)-X_2f(X_1,Y,Z)\\
&=\{\{X_1X_2,Y\},Z\}- X_1\{\{X_2,Y\},Z\}- X_2\{\{X_1,Y\},Z\}\\
&=\{X_1,Y\}\{X_2,Z\}+\{X_1,Z\}\{X_2,Y\},
\end{align*}
which is customary.
Farkas called this process a {\it customarization}~\cite{Farkas98}, it is an analogue
of the linearization process for associative algebras.
The arguments of \cite{Farkas98} actually prove the following.
\begin{Th}[\cite{Farkas98}]\label{TFarkas}
Suppose that a Poisson algebra $A$ over an arbitrary field
satisfies a nontrivial {\bf multilinear} Poisson identity.
Then $A$ satisfies a nontrivial customary identity.
\end{Th}
Let us explain why we {\it need all polynomials to be multilinear in case of positive characteristic} $p$.
The linearization process is simply not working for Poisson algebras in positive characteristic
as it does for associative and Lie algebras.
For example, the Poisson identity $\{x,y\}^p\equiv 0$ is given by a nonzero element of
the free Poisson algebra $F(X)$.
Observe that its full linearization is trivial:
$$\sum_{\sigma,\pi\in S_p }\{x_{\sigma(1)},y_{\pi(1)}\}\cdots \{x_{\sigma(p)},y_{\pi(p)}\}=
p!\sum_{\pi\in S_p }\{x_{1},y_{\pi(1)}\}\cdots \{x_{p},y_{\pi(p)}\}=0.
$$
Moreover, let us check that any truncated symmetric algebra $\SSS (L)$
satisfies the identity $\{x,y\}^p\equiv 0$. Indeed, let $a,b\in
\SSS (L)$, then $\{a,b\}$ is a truncated polynomial without constant
term, its $p$th power is zero by the Frobenius rule
$(v+w)^p=v^p+w^p$. Thus, it does not make sense to study nonlinear
Poisson identities for truncated symmetric algebras.
\medskip

In the theory of Poisson PI-algebras, the analogue of the standard polynomial is (\cite{Farkas98},~\cite{Farkas99}):
$$ \St_{2n}=\St_{2n}(x_1,\dots,x_{2n})=\sum_{\sigma\in T_{2n}}(-1)^\sigma \{x_{\sigma (1)},x_{\sigma (2)}\}\cdots
\{x_{\sigma (2n-1)},x_{\sigma (2n)}\}. $$
This is a customary polynomial, skewsymmetric in all variables \cite{Farkas98}.
One has the following fact similar to the theory of associative algebras.

\begin{Th}[\cite{MiPeRe}]
In case of zero characteristic, any Poisson PI-algebra satisfies an identity $(\St_{2n})^m\equiv 0$, for some integers $n,m$.
\end{Th}

Another important fact on the standard identity is as follows.
\begin{Lemma}[\cite{Farkas98}]
Let $A$ be a Poisson algebra over an arbitrary field $K$ and $A$ is $k$-generated as an associative algebra.
Then it satisfies the standard identity
$\St_{2m}\equiv 0$,
whenever $2m>k$.
\end{Lemma}

\section{Identical relations of group rings and enveloping algebras}
\label{Sidentities}

In this section we present a motivation for our research project.
Namely, we shortly review results on existence of nontrivial
polynomial identities in enveloping algebras and group rings.

\subsection{Identities of group rings}

Passman obtained necessary and sufficient conditions for a group ring $K[G]$ to
satisfy a nontrivial polynomial identity over a field $K$ of an arbitrary characteristic.
A group $G$ is said {\it p-abelian} if $G$ is abelian in case $p=0$ and, in case $p>0$, $G'$,
the commutator subgroup of $G$, is a finite $p$-group.

\begin{Th}[\cite{Pas72}]
\label{TKG}
The group  algebra $K[G]$ of a group $G$ satisfies
a nontrivial polynomial identity if and only if the following conditions are satisfied.
\begin{enumerate}
  \item
    There exists a subgroup  $A\subseteq G$ of finite index;
  \item
    $A$ is $p$-abelian.
\end{enumerate}
\end{Th}

All our associative algebras are with unity.
Recall the notions of a (strong) Lie nilpotence and (strong) solvability for associative algebras.
Let $A$ be an associative algebra, and $A^{(-)}$ the related Lie algebra.
Consider its {\it lower central series}:
$\gamma_1(A)=A$, $\gamma_{i+1}(A)=[\gamma_i(A),A]$, $i\ge 1$.
Algebra $A$ is said {\it Lie nilpotent} of class $s$ iff $\gamma_{s+1}(A)=0$ and $\gamma_s(A)\ne 0$.
Also consider {\it upper Lie powers}: $A^{(0)}=A$ and $A^{(n+1)}=[A^{(n)},A]A$, $n\ge 0$
(we use the shifted enumeration in comparison with~\cite{PaPaSe73,RiSh95} because
$\{A^{(n)}\,|\, n\ge 0\}$ is a filtration, a proof is similar to that of Lemma~\ref{RiRj=Ri+j}).
Now, $A$ is {\it strongly Lie nilpotent} of class $s$ iff $A^{(s)}=0$ and $A^{(s-1)}\ne 0$.
One defines the {\it derived series}: $\delta_{0}(A)=A$, $\delta_{i+1}(A)=[\delta_{i}(A),\delta_{i}(A)]$, $i\ge 0$.
Algebra $A$ is {\it solvable} of length $s$ iff $\delta_{s}(A)=0$ and $\delta_{s-1}(A)\ne 0$.
Consider also the {\it upper derived series}: $\tilde\delta_{0}(A)=A$,
$\tilde\delta_{i+1}(A)=[\tilde\delta_{i}(A),\tilde\delta_{i}(A)]A$, $i\ge 0$.
Now, $A$ is {\it strongly solvable} of length $s$ iff
$\tilde\delta_{s}(A)=0$ and $\tilde\delta_{s-1}(A)\ne 0$.

Passi, Passman and Sehgal
specified the Lie nilpotence and solvability of $K[G]$~\cite{PaPaSe73}.
\begin{Th}[\cite{PaPaSe73}]\label{TKGNS}
Let $K[G]$ be the group ring of a group $G$ over a field $K$, $\ch K=p\geq 0$. Then
\begin{enumerate}
  \item
    $K[G]$ is Lie nilpotent if and only if $G$ is $p$-abelian and nilpotent;
  \item
    $K[G]$ is solvable if and only if $G$ is $p$-abelian, for $p\ne 2$;
  \item
    $K[G]$ is solvable if and only if $G$ has a $2$-abelian subgroup of index at most $2$, for $p=2$.
\end{enumerate}
\end{Th}

Using the upper Lie powers, one defines {\it Lie dimension subgroups} of a group (our enumeration is shifted)~\cite{Passi}:
$$
D_{(n),K}(G)=G\cap (1+K[G]^{(n)}), \qquad n\ge 0.
$$
See another description~\cite{BhaPas92}:
\begin{equation}\label{dim-subgroups}
D_{(n),K}(G)=\prod_{(i-1)p^k\ge n} \gamma_i(G)^{p^k},\quad n\ge 0.
\end{equation}
There is a formula for the Lie nilpotency class of a modular group ring.
\begin{Th}[\cite{BhaPas92}]\label{Tgr-nilp}
Let $G$ be a group, $K$ a field of characteristic $p>3$ such
that the group ring $K[G]$ is Lie nilpotent.
The Lie nilpotency class of $K[G]$ coincides with its strong Lie nilpotency class and is equal to
$$
1+(p-1)\sum_{m\ge 1}m \log_p \big|D_{(m),K}(G):D_{(m+1),K}(G)\big|.
$$
\end{Th}

\subsection{Identities of enveloping algebras}
Latyshev proved that the universal enveloping algebra of a finite
dimensional Lie algebra over a field of characteristic zero satisfies a nontrivial
polynomial identity if and only if the Lie algebra is abelian~\cite{Lat63}.
Later Bahturin noticed that the condition of a finite dimensionality is inessential (see e.g.~\cite{Ba}).

Bahturin settled a similar problem on the existence of  a
nontrivial identity for the universal enveloping algebra over a
field of positive characteristic~\cite{Ba74}.
Also, Bahturin found necessary and sufficient conditions for the universal enveloping algebra
of a Lie superalgebra over a field of characteristic zero to satisfy a non-trivial polynomial identity~\cite{Ba85}.
\medskip

Passman~\cite{Pas90} and Petrogradsky~\cite{Pe91} described
restricted Lie algebras $L$ whose restricted enveloping
algebra $u(L)$ satisfies a nontrivial polynomial identity.

\begin{Th}[\cite{Pas90},~\cite{Pe91}]
\label{TuL}
Let $L$ be a Lie $p$-algebra.
The restricted enveloping algebra $u(L)$ satisfies
a nontrivial polynomial identity if and only if
there exist restricted ideals $Q\subseteq H\subseteq L$ such that
 \begin{enumerate}
   \item
     $\dim L/H<\infty$, $\dim Q<\infty$;
   \item
     $H/Q$ is abelian;
   \item
     $Q$ is abelian and has a nilpotent $p$-mapping.
 \end{enumerate}
\end{Th}

Riley and Shalev determined when $u(L)$
is Lie nilpotent, solvable (for $p>2$), or satisfies the Engel condition~\cite{RiSh93}.

\begin{Th}[\cite{RiSh93}]\label{uLNS}
Let $u(L)$ be the restricted enveloping algebra of a restricted Lie algebra $L$ over a field $K$ of characteristic $p>0$.
\begin{enumerate}
  \item
    $u(L)$ is Lie nilpotent if and only if $L$ is nilpotent and $L^2$ is finite dimensional and $p$-nilpotent;
  \item
    $u(L)$ is $n$-Engel for some $n$ if and only if $L$ is nilpotent, $L^2$ is $p$-nilpotent, and $L$ has a restricted ideal $A$ such that both
    $L/A$ and $A^2$ are finite dimensional.
  \item
    $u(L)$ is solvable if and only if $L^2$ is finite dimensional and $p$-nilpotent, for $p\neq 2$.
    \end{enumerate}
\end{Th}

Let $L$ be a restricted Lie algebra, $\ch K=p>0$.
Similar to the dimension subgroups, using upper Lie powers,
Riley and Shalev defined {\it Lie dimension subalgebras}~\cite{RiSh95}:
$$
D_{(n)}(L)=L\cap u(L)^{(n)}, \qquad n\ge 0.
$$
(recall that our enumeration is shifted).
They gave another description~\cite{RiSh95}:
\begin{equation}\label{dim-subalg}
D_{(n)}(L)=\sum_{(i-1)p^k\ge n} \gamma_i(L)^{[p^k]},\quad n\ge 0.
\end{equation}
Siciliano proved~\cite{Sic06} that in case $p>2$,
the strong solvability of the restricted enveloping algebra $u(L)$ is equivalent to its solvability.
Moreover, the strong solvability in case $p=2$ is described by the same conditions of Part 3 of Theorem~\ref{uLNS}.
Also, in case $p=2$ he observed an example of the restricted enveloping algebra $u(L)$ that is
solvable but not strongly solvable.

The following is an analogue of results on the Lie nilpotency classes of group rings (Theorem~\ref{Tgr-nilp}).
\begin{Th}[\cite{RiSh95}]\label{Tu-nilp}
Let $L$ be a restricted Lie algebra over a field $K$ of characteristic $p$ such
that $u(L)$ is Lie nilpotent. Then
\begin{enumerate}
\item The strong Lie nilpotency class of $u(L)$ is equal to
$$ 1+(p-1)\sum_{m\ge 1}m \dim (D_{(m)}(L)/D_{(m+1)}(L)); $$
\item In case $p>3$, the Lie nilpotency class coincides with the strong Lie nilpotency class.
\end{enumerate}
\end{Th}

The solvability of restricted enveloping algebras in case of characteristic 2 was only recently settled in~\cite{SiUs13}.
Lie nilpotence, solvability, and other non-matrix identities for (restricted) enveloping algebras of (restricted)
Lie (super)algebras are studied in~\cite{BRU,Pe92,SicSpi06,SiUs13.2,SiUs15.2,Use13,Use13.2}.
More on derived lengths, Lie nilpotency classes for $u(L)$, or identities for symmetric elements
of $u(L)$, etc., see the survey~\cite{SiUs15}.

More general cases of (restricted) enveloping algebras for (color) Lie $p$-(super)algebras are treated in~\cite{BMPZ}.
Further developments have been obtained for smash products  $U(L)\#K[G]$ and $u(L)\#K[G]$,
where a group $G$ acts by automorphisms on a (restricted) Lie algebra $L$~\cite{BaPe02}.
Identities of smash products $U(L)\#K[G]$, where $L$ is a Lie superalgebra in characteristic zero were studied in~\cite{Koc03}.
The results on identities of smash products are of interest because they combine, as particular cases,
both, the results on identities of group ring and enveloping algebras.

\subsection{Identities of symmetric Poisson algebras}

The following result is an analogue of the classical Amitsur-Levitzki theorem on identities of matrix algebras.
Kostant used another terminology,
but as observed Farkas~\cite{Farkas99}, this is a result on identities of symmetric Poisson algebras.

\begin{Th} [\cite{Kos81,Farkas99}]\label{TKost}
Let $L$ be a finite dimensional Lie algebra over a field of characteristic zero.
The symmetric algebra $S(L)$ satisfies the standard Poisson identity $\St_ {2d}\equiv 0$
as soon as $2d$ exceeds the dimension of a maximal coadjoint orbit of $L$.
\end{Th}

The Lie nilpotence of class 2 of symmetric algebras $S(L)$, where $L$ is a Lie superalgebra,
was specified by Shestakov. The next statement follows from Theorem 4 and Theorem 5 of~\cite{Shestakov93}.
\begin{Th}[\cite{Shestakov93}]\label{TSh93}
The symmetric algebra $S(L)$ of a Lie algebra $L$, over a field $K$, satisfies
the identity $\{x,\{y,z\}\}\equiv 0$ if and only if $L$ is abelian.
\end{Th}

Farkas proved the following statement that generalizes Kostant's Theorem~\ref{TKost}.

\begin{Th}[\cite{Farkas99}] \label{TSymm0}
Let $L$ be a Lie algebra over a field of characteristic zero.
The symmetric algebra $S(L)$ satisfies a
nontrivial Poisson identity if and only if $L$ contains an abelian
subalgebra of finite codimension.
\end{Th}
Giambruno and Petrogradsky extended this result to an arbitrary characteristic~\cite{GiPe06}.

\begin{Th}[\cite{GiPe06}]\label{TSymm_p}
Let $L$ be a Lie algebra over an arbitrary field.
The symmetric algebra $S(L)$ satisfies a nontrivial {\em multilinear}
Poisson identity if and only if $L$ contains an abelian subalgebra of finite codimension.
\end{Th}

The following result was obtained for the truncated symmetric algebra $\SSS(L)$
of a restricted Lie algebra $L$ (see definitions~\cite{JacLie}).
\begin{Th}[\cite{GiPe06}]\label{Tmain}
Let $L$ be a restricted Lie algebra.
The truncated symmetric algebra
$\SSS(L)$ satisfies a nontrivial {\em multilinear} Poisson identity if and only if
there exists a restricted ideal $H\subseteq L$ such that
 \begin{enumerate}
   \item
     $\dim L/H<\infty$,
   \item
     $\dim H^2<\infty$;
   \item
     $H$ is nilpotent of class 2.
 \end{enumerate}
\end{Th}

The present paper is heavily based on this result. Remark that
the identities of the (strong) Lie nilpotence and (strong) solvability are multilinear.
We shall apply this result in a more precise form given by Theorem~\ref{Treduction}.

\section{Main results: Lie identities of symmetric algebras}
\label{Smain}
In this section we formulate our main results.
\subsection{Lie nilpotence of truncated symmetric algebras $\SSS(L)$}
Let $R$ be an arbitrary Poisson algebra.
Consider the {\em lower central series} of $R$ as a Lie algebra,
i.e., $\gamma_1(R)=R$ and $\gamma_{n+1}(R)=\{\gamma_n(R),R\}, n\geq1$.
We say that $R$ is {\it Lie nilpotent of class $s$} if and only if  $\gamma_{s+1}(R)=0$ but $\gamma_s(R)\neq 0$.
Clearly, the condition $\gamma_{s+1}(R)=0$ is equivalent to the {\it identity of Lie nilpotence of class} $s$:
$$\{\ldots\{\{X_0,X_1\},X_2\},\ldots,X_s\}\equiv 0.$$

Similar to the associative case, one defines {\em upper Lie powers} as follows.
At each step we take ideals generated by commutators, namely,
set $R^{(0)}=R$ and $R^{(n)}=\{R^{(n-1)},R\}\cdot R$ for all $n\geq 1$
(our enumeration is shifted, a justification is that
$\{R^{(n)}\,|\, n\ge 0\}$ is a filtration, see Lemma~\ref{RiRj=Ri+j}).
We call a Poisson algebra $R$ {\em strongly Lie nilpotent of class s}
if and only if $R^{(s)}= 0$ but $R^{(s-1)}= 0$.
The condition $R^{(s)}=0$ is equivalent to the identical relation of the {\em strong Lie nilpotence of class~$s$}:
\begin{equation*}
\{\{\ldots\{\{X_0,X_1\}\cdot Y_1,X_2\}\cdot Y_2,\ldots, X_{s-1}\}\cdot Y_{s-1},X_s\}\equiv 0.
\end{equation*}
 Observe that
\begin{equation}
\gamma_n(R)\subseteq R^{(n-1)},\ n\geq1.\label{gamma(R)em R(n-1)}
\end{equation}
Thus, the strong Lie nilpotence of class $s$ implies the Lie nilpotence of class at most $s$.
The Lie nilpotence of class $1$ is equivalent to the strong Lie nilpotence of class $1$ and equivalent to the fact
that $R$ is abelian as a Lie algebra.
\medskip

The following is the first main result of the paper.
\begin{Th}\label{Tnilp}
Let $L$ be a Lie algebra over a field of positive characteristic $p$. Consider
its truncated symmetric Poisson algebra $\SSS(L)$. The following conditions are equivalent:
\begin{enumerate}
\item $\SSS(L)$ is strongly Lie nilpotent;
\item $\SSS(L)$ is Lie nilpotent;
\item $L$ is nilpotent and $\dim L^2<\infty$.
\end{enumerate}
\end{Th}
Let $L$ be a Lie algebra over an arbitrary field $K$ ($\ch K=p$).
Using the upper Lie powers, define
{\it Poisson dimension subalgebras} ({\it truncated Poisson dimension subalgebras}, respectively)  of $L$:
\begin{alignat*}{2}
D^{S}_{(n)}(L)&=L\cap (S (L))^{(n)},  &\qquad n&\ge 0;\\
D^{s}_{(n)}(L)&=L\cap (\SSS (L))^{(n)},& \qquad n&\ge 0.
\end{alignat*}
By Corollary~\ref{C_dim_subalg} (Claim~3 of Lemma~\ref{Lgamma}, respectively),
we obtain a description of these subalgebras similar to that
for the group rings~\eqref{dim-subgroups} and the restricted enveloping algebras~\eqref{dim-subalg}
(compare with the first term of that products):
\begin{alignat*}{2}
D^S_{(n)}(L)&=\gamma_{n+1}(L),&\qquad n&\ge 0;\\
D^s_{(n)}(L)&=\gamma_{n+1}(L),&\qquad n&\ge 0.
\end{alignat*}
We compute the classes of Lie nilpotence and strong Lie nilpotence.
Our formula is an analogue of the formulas known for group rings (Theorem~\ref{Tgr-nilp}) and restricted enveloping algebras
(Theorem~\ref{Tu-nilp}). The analogy is better seen in terms of truncated Poisson dimension subalgebras.

\begin{Th}\label{Tclasses}
Let $L$ be a Lie algebra over a field of positive characteristic $p>3$,
such that the truncated symmetric Poisson algebra $\SSS(L)$ is Lie nilpotent.
The following numbers are equal:
\begin{enumerate}
\item the strong Lie nilpotency class of $\SSS(L)$;
\item the Lie nilpotency class of  $\SSS(L)$;
\item $$1+(p-1)\sum_{n\geq1}n\cdot\dim(\gamma_{n+1}(L)/\gamma_{n+2}(L)).$$
\end{enumerate}
In cases $p=2,3$, the numbers 1) and 3) remain equal.
\end{Th}
For cases $p=2,3$, the number above yields an upper bound for the Lie nilpotency class.
Also, we have a lower bound for the Lie nilpotency class, $L$ being non-abelian (Lemma~\ref{L_cl_lower}):
$$2+(p-1)\sum_{n\geq2}(n-1)\cdot\dim(\gamma_{n+1}(L)/\gamma_{n+2}(L)).$$

\subsection{Solvability of truncated symmetric algebras $\SSS(L)$}
Let $R$ be a  Poisson algebra.
Consider its {\it derived series} as a Lie algebra:
$\delta_0(R)=R$, $\delta_{n+1}(R)=\{\delta_n(R),\delta_n(R)\}$, $n\geq 0$.
Polynomials of {\it solvability} are defined by recursion: $\delta_1(X_1,X_2)=\{X_1,X_2\}$ and
\begin{equation*}
\delta_{n+1}(X_1,X_2,\ldots,X_{2^{n+1}})
=\big\{\delta_{n}(X_1,\ldots,X_{2^n}),\delta_n(X_{2^{n}+1},\ldots,X_{2^{n+1}})\big\}, \quad n\ge 1.
\end{equation*}
A Poisson algebra $R$ is \textit{solvable of length $s$} if,
and only if, $\delta_{s}(R)=0$ and $\delta_{s-1}(R)\ne 0$,
equivalently, that $R$ satisfies the above identity of Lie solvability
$\delta_s(\ldots)\equiv 0$, where $s$ is minimal.

Define the {\it upper derived series}:
$\tilde{\delta}_0(R)=R$ and $\tilde{\delta}_{n+1}(R)=\{\tilde{\delta}_{n}(R),\tilde{\delta}_{n}(R)\}\cdot R$, $n\geq 0$.
Define polynomials of the {\it strong solvability} by
$\tilde{\delta}_1(X_1,X_2,Y_1)=\{X_1,X_2\}\cdot Y_1$, and
\begin{multline*}
\tilde{\delta}_{n+1}(X_1,\ldots,X_{2^{n+1}},Y_1,\ldots,Y_{2^{n+1}-1}) \\
=\Big\{\tilde{\delta}_{n}(X_1,\ldots,X_{2^{n}},Y_1,\ldots,Y_{2^{n}-1}),\qquad\qquad\qquad\\
 \tilde{\delta}_{n}(X_{2^{n}+1},\ldots,X_{2^{n+1}},Y_{2^{n}},\ldots,Y_{2^{n+1}-2})\Big\} \cdot Y_{2^{n+1}-1},
 \quad n\ge 1.
\end{multline*}

A Poisson algebra $R$ is called {\it strongly solvable of length} $s$ if and only if
$\tilde{\delta}_{s}(R)=0$ and $\tilde{\delta}_{s-1}(R)\ne 0$,
or equivalently that $R$ satisfies the above identity $\tilde\delta_s(\ldots)\equiv 0$, where $s$ is minimal.
Observe that
\begin{equation}\label{deltindeltb}
\delta_{s}(R)\subseteq\tilde{\delta}_{s}(R),\qquad s\geq 0.
\end{equation}
Thus, the strong solvability of length $s$ implies the solvability of length at most $s$.
The solvability of length $1$ is equivalent to the strong solvability of length $1$ and equivalent to the fact
that $R$ is abelian as a Lie algebra.
\medskip

The following is the second main result of the paper.
\begin{Th}\label{Tsolv}
Let $L$ be a Lie algebra over a field of positive characteristic $p\geq 3$. Consider
its truncated symmetric Poisson algebra $\SSS(L)$. The following conditions are equivalent:
\begin{enumerate}
\item $\SSS(L)$ is strongly solvable;
\item $\SSS(L)$ is solvable;
\item $L$ is solvable and $\dim L^2<\infty$.
\end{enumerate}
In case $p=2$, conditions 1) and 3) remain equivalent.
\end{Th}

Observe that, the description of solvable group rings in characteristic 2
looks very nice (Theorem~\ref{TKGNS}).
But the answer to a similar question for the restricted enveloping algebras is rather complicated
and was obtained only recently~\cite{SiUs13}.

The problem of solvability of $\SSS(L)$ in case $\ch K=2$ is open.
As a first approach, we show that the situation is different to other characteristics.
Namely, in case $\ch K=2$, we obtain two examples of truncated symmetric Poisson algebras that are solvable
but not strongly solvable, see Lemma~\ref{L22} and Lemma~\ref{L22b}.
A close fact is that the Hamiltonian algebras $\mathbf{H}_2(K)$ and $\mathbf{h}_2(K)$ are solvable
but not strongly solvable in case $\ch K=2$ (Lemma~\ref{L23}).
This is an analogue of a well-known fact that the matrix ring $\mathrm{M}_2(K)$ of $2\times 2$ matrices over
a field $K$, $\ch K=2$, is solvable but not strongly solvable.

\subsection{Nilpotency and solvability of symmetric algebras $S(L)$}
Finally, we prove the following extension of the result of Shestakov (Theorem~\ref{TSh93}).

\begin{Th}\label{Pe16}
Let $L$ be a Lie algebra over a field $K$, and $S(L)$ its symmetric Poisson algebra.
The following conditions are equivalent:
\begin{enumerate}
\item $L$ is abelian;
\item $S(L)$ is strongly Lie nilpotent;
\item $S(L)$ is Lie nilpotent;
\item $S(L)$ is strongly solvable;
\item $S(L)$ is solvable (here assume that $\ch K\ne 2$).
\end{enumerate}
\end{Th}
In case $\ch K=2$, the solvability of the symmetric Poisson algebra $S(L)$ is an open question.
Two examples of Lie algebras mentioned above also
yield solvable symmetric algebras which are not strongly solvable (Lemma~\ref{LS2solvA} and Lemma~\ref{LS2solvB}).
\subsection*{Conjecture}
Let $L$ be a Lie algebra over a field $K$, $\ch K=2$.
The symmetric algebra $S(L)$ (probably, the truncated symmetric algebra $\SSS(L)$ as well)
is solvable if and only if $L=\langle x\rangle_K \oplus A$, where $A$ is an abelian ideal on which
$\ad x$ acts algebraically.
\subsection*{Remark}
Formally, our statements on ordinary Lie nilpotency and solvability are concerned
only with the {\it Lie structure} of Poisson algebras $\SSS(L)$ and $S(L)$.
But our proof heavily relies on Theorem~\ref{Treduction} of~\cite{GiPe06}, which in turn
uses the existence of a nontrivial customary identity given by Theorem~\ref{TFarkas} (Farkas~\cite{Farkas98}).
In this way, we need the {\it Poisson structure} of our algebras to prove our results.
We do not see ways to prove them using the theory of Lie identical relations only.

\section{Delta-sets and multilinear Poisson identical relations}

The goal of this section is to present a reduction step given by Theorem~\ref{Treduction}.
{\it Delta-sets} in groups were introduced by Passman to study identities in the group rings~\cite{Pas72}.
In case of Lie algebras, the delta-sets were introduced by Bahturin to study identical relations of the universal enveloping algebras~\cite{Ba74}.
Let $L$ be a Lie algebra, one defines the delta-sets as sets of elements of {\it finite width} as follows:
\begin{align*}
\Delta_n(L)&=\{x\in L\mid\dim [L,x]\leq n\}, \qquad n\geq0;\\
\Delta(L) &=\mathop{\cup}\limits_{n=0}^{\infty}\Delta_n(L)=\{x\in L\mid\dim [L,x]<\infty\}.
\end{align*}

Note that $\Delta_n(L)$, $n\ge 0$, is not a subalgebra or even a subspace in a general case.
The basic properties of the delta-sets are as follows.

\begin{Lemma}[\cite{BMPZ,Pe92}]
\label{PropDel}
Let $L$ be a (restricted) Lie algebra, $n,m\ge 0$.
 \begin{enumerate}
  \item
    $\Delta_n(L)$ is invariant under scalar multiplication;
  \item
    let $x\in\Delta_n(L)$, $y\in\Delta_m(L)$, then
     $\alpha x+\beta y\in\Delta_{n+m}(L)$, where  $\alpha,\beta\in K$;
  \item
    let $x\in\Delta_n(L)$, $y\in L$, then
     $[x,y]\in\Delta_{2n}(L)$;
  \item
    let $x\in\Delta_n(L)$ and $L$ a restricted
    Lie algebra, then $x^{[p]}\in\Delta_{n}(L)$;
  \item
    $\Delta(L)$  is a (restricted) ideal of $L$.
\end{enumerate}
\end{Lemma}
\begin{Lemma}[\cite{RiSh93}]\label{DeltaI}
Let $L$ be a Lie algebra.
\begin{enumerate}
\item
     if $I$ is a finite dimensional ideal of $L$, then $\Delta(L/I)=(\Delta(L)+I)/I$;
\item
     if $H$ is a subalgebra of finite codimension in $L$, then $\Delta(H)=\Delta(L)\cap H$.
\end{enumerate}
\end{Lemma}

Suppose that $W$ is a subset in a $K$-vector space $V$.
We say that $W$ has finite codimension in $V$ if there exist
$v_1,\dots,v_m\in V$ such that
$V=\{w+\lambda_1v_1+\dots+\lambda_m v_m \,|\,w\in W,\ \lambda_1,\ldots, \lambda_m\in K\}$.
Let $m$ be the minimum integer with such property, then we denote $\dim V/W=m$.
We also introduce a notation
$m\cdot W=\{w_1+\cdots+w_m\,|\, w_i\in W\}$, where $m\in\N$.
\begin{Lemma}[{\cite[Lemma~6.3]{BaPe02}}]
\label{Ldilat}
Let $V$ be a $K$-vector space. Suppose that a subset
$T\subseteq V$ is stable under multiplication by scalars and $\dim V/T\le n$.
Then one obtains its linear span as follows: $\langle T\rangle_K=4^n\cdot T$.
\end{Lemma}

We need a result on bilinear maps.
\begin{Th}[P.M. Neumann,~\cite{Ba}]
\label{TNeu}
Let $U,V,W$ be vector spaces over a field $K$ and
$\varphi:U\times V \to W$ a bilinear map. Suppose that for all
$u\in U$ and $v\in  V$, $\dim \varphi (u,V)\le m$ and
$\dim \varphi (U,v)\le l$. Then $\dim \langle\varphi (U,V)\rangle_K\leq ml$.
\end{Th}

The following crucial result was proved in case of restricted Lie algebras~\cite{GiPe06},
but its proof remains valid for truncated symmetric algebras as well.
\begin{Th}[\cite{GiPe06}]
\label{TtrPoisson}
Let $L$ be a Lie algebra.
Suppose that the symmetric algebra $S(L)$
(or the truncated symmetric algebra $\SSS(L)$) satisfies a multilinear Poisson identity.
There exist integers $n,N$ such that $\dim L/\Delta_N(L)<n$.
\end{Th}
It yields the following reduction step, which is actually contained in proofs of~\cite{GiPe06}.
\begin{Th}[\cite{GiPe06}]
\label{Treduction}
Let $L$ be a Lie algebra such that
the symmetric algebra $S(L)$
(or the truncated symmetric algebra $\SSS(L)$) satisfies a multilinear Poisson identity.
Denote $\Delta=\Delta(L)$.
There exist integers $n,M$ such that
\begin{enumerate}
\item $\Delta=\Delta_M(L)$;
\item $\dim L/\Delta<n$;
\item $\dim \Delta^2\le M^2 $.
\end{enumerate}
\end{Th}
\begin{proof}
By Theorem~\ref{TtrPoisson}, we have $\dim L/\Delta_N(L)<n$.
By Lemma~\ref{PropDel} and Lemma~\ref{Ldilat}, to obtain the linear span $\langle \Delta_N(L)\rangle_K$
it is sufficient to consider $4^n$-fold sums.
Applying Lemma~\ref{PropDel}, we get $\langle \Delta_N(L)\rangle_K\subseteq \Delta_{4^n N}(L)$.
By virtue of finite codimension, we get $\Delta=\Delta_{M}(L)$ for some integer $M$.
Finally, $\dim [\Delta,\Delta]\le M^2$ by Theorem~\ref{TNeu}.
\end{proof}

\section{Filtrations of Lie algebras and symmetric algebras}
\label{Sfiltr}

Let $L$ be a Lie algebra over a field $K$ of positive characteristic $p$.
A {\it filtration} of the truncated symmetric Poisson algebra $\SSS(L)$ is a descending sequence of subspaces
\begin{align*}
&\SSS(L) = E_0 \supseteq E_1 \supseteq \cdots \supseteq E_n \supseteq\cdots,\qquad \mathop{\cap}\limits_{n= 0}^\infty E_n=0, \\
&E_i\cdot E_j\subseteq E_{i+j},\quad \{E_i,E_j\}\subseteq E_{i+j}, \qquad i,j\geq0.
\end{align*}
By setting $L_n=L\cap E_n$, $n\geq 0$, one gets naturally a sequence of subspaces:
\begin{equation*}
L=L_0\supseteq L_1\supseteq \cdots\supseteq L_n\supseteq\cdots,\qquad \mathop{\cap}\limits_{n=0}^\infty L_n=0,\qquad
\{L_i,L_j\}\subseteq L_{i+j},\qquad  i,j\geq0.
\end{equation*}
The last inclusion follows by definition of the product $\{\ ,\ \}$ in $\SSS(L)$.
We call a sequence of subspaces above a \textit{filtration} of $L$.
We have just seen that a filtration of $\SSS(L)$ gives rise to a filtration of $L$.
One has a converse statement.
\begin{Lemma}
Let $\{L_n \,|\, n\geq0\}$ be a filtration of a Lie algebra $L$ over a field $K$ of characteristic $p>0$.
For each $0\ne x\in L$ define the {\em height} of $x$,
denoted by $\nu(x)$, as the largest subscript $n$ such that $x \in L_n$.
Define
\begin{equation}\label{En}
E_n =\big\langle y_1y_2\cdots y_l \,\big\vert\,  y_i\in L,\ \nu (y_1)+\nu (y_2)+\cdots +\nu(y_l)\geq n\big\rangle_K,\quad n\geq 0.
\end{equation}
Then, $\{E_n\,|\, n\geq0\}$ is a filtration of $\SSS(L)$.
\end{Lemma}
\begin{proof}
The inclusion $E_n\cdot E_m\subseteq E_{n+m}$, $n,m\ge 0$, is clear.
Let us prove that $\{E_n,E_m\}\subseteq E_{n+m}$, $n,m\ge 0$.
Let $a=x_1\cdots x_l\in E_n$, and $b=y_1\cdots y_k\in E_m$.
By the Leibnitz rule,
$\{a,b\}=\sum_{i,j}x_1\cdots\widehat{x_i}\cdots x_l\cdot y_1\cdots\widehat{y_j}\cdots y_k\{x_i,y_j\}$.
Since $\nu(\{x_i,y_j\})\geq \nu(x_i)+\nu(y_j)$, we conclude that $\{a,b\}\in E_{m+n}$.
\end{proof}
We call $\{E_n \,|\, n\geq0\}$ the {\it induced filtration}.
A close relationship between these filtrations is demonstrated by the next result.
A similar relationship was used in~\cite{Pe11}.
\begin{Lemma}\label{Filt}
Let $\{L_n \,|\, n\geq0\}$ be a filtration of a Lie algebra $L$ over a field $K$ of characteristic $p>0$,
$\{x_i \,|\, i\in I\}$ an ordered basis of $L$ chosen so that $L_n=\langle x_i \,|\, \nu (x_i)\geq n\rangle_K$, $n\geq0$.
Let $\{E_n \,|\, n\geq0\}$ be the induced filtration of $\SSS(L)$.
We have the following statements for $n\geq0$:
\begin{enumerate}
\item
$E_n=\Big\langle \eta= x_{i_1}^{a_1}x_{i_2}^{a_2}\cdots x_{i_l}^{a_l}\,\Big\vert\,
\nu(\eta)=\sum\limits_{j=1}^l a_j\nu (x_{i_j}) \geq n,\
i_1<\cdots< i_l,\
0\leq a_j< p  ,\ l\geq 0
\Big\rangle_K; $
\item $\{\eta\, |\, \nu(\eta)=n\}$ is a basis of $E_n$ modulo $E_{n+1}$;
\item the set of all such monomials $\eta$ forms a basis of $\SSS(L)$;
\item $L_n=L\cap E_n$ (i.e., we return to the initial filtration).
\end{enumerate}
\end{Lemma}
\begin{proof}
Clearly, all monomials $\eta$ above belong to $E_n$.
Consider $0\ne y\in L_m$, $m\ge 0$, present it as a finite sum $y=\sum_{\nu(x_i)\ge m} \lambda_i x_i$, $\lambda_i\in K$.
Apply such expansions for all $y_i$ in product~\eqref{En} and reorder multiplicands of the resulting monomials.
In this way we present products~\eqref{En} as linear combinations of monomials $\{\eta\,|\,\nu(\eta)\ge n\}$,
the latter being linearly independent as a canonical basis of $\SSS(L)$.
The first claim is proved.
The remaining claims follow.
\end{proof}
\begin{Lemma}\label{RiRj=Ri+j}
Let $R$ be a Poisson algebra.
The upper Lie powers $\{R^{(n)}\,|\, n\ge 0\}$ form a filtration.
\end{Lemma}
\begin{proof}
First, we prove that $R^{(i)}\cdot R^{(j)}\subseteq R^{(i+j)}$ for all $i,j\ge 0$ by induction on $j$.
Let $j=0$, then $R^{(0)}=R$ and there is nothing to prove.
Suppose that the statement is valid for $j-1\geq 0$.
We use the Leibnitz rule in form $a\{b,c\}=\{ab,c\}-b\{a,c\}$ and apply the inductive assumption:
\begin{align*}
R^{(i)}\cdot R^{(j)}&= R^{(i)}\{R^{(j-1)},R\}R
         \subseteq \{R^{(i)} R^{(j-1)}, R\}R +R^{(j-1)}\{R^{(i)},R\}R\\
&\subseteq \{R^{(i+j-1)},R\}R+R^{(j-1)}R^{(i+1)}
         \subseteq R^{(i+j)}.
\end{align*}

Second, we prove that $\{R^{(i)},R^{(j)}\}\subseteq R^{(i+j)}$ for all $i,j\ge 0$, by induction on $j$.
Let $j=0$, then $\{R^{(i)},R^{(0)}\}=\{R^{(i)},R\}\subseteq R^{(i+1)}\subseteq R^{(i)}$.
Suppose that the statement is valid for $j-1\geq 0$.
We use the Leibnitz rule, the Jacobi identity, and the inclusion above:
\begin{align*}
\{R^{(i)}, R^{(j)}\}
&=\{R^{(i)},\{R^{(j-1)},R\}\cdot R\}
 \subseteq \{R^{(j-1)}, R\}\cdot \{R^{(i)},R\}+\{R^{(i)},\{R^{(j-1)},R\}\}\cdot R \\
&\subseteq  R^{(j)}\cdot R^{(i+1)}
        +\left(\{\{R^{(i)},R^{(j-1)}\},R\} + \{ R^{(j-1)},\{ R^{(i)},R \}\} \right)\cdot R\\
&\subseteq R^{(i+j+1)}+\left(\{ R^{(i+j-1)},R\}+ \{R^{(j-1)},R^{(i+1)}\}\right)\cdot R\\
&\subseteq R^{(i+j+1)}+R^{(i+j)}\cdot R
 \subseteq R^{(i+j)}. \qedhere
\end{align*}
\end{proof}

\section{Lie nilpotence of truncated symmetric algebras $\SSS(L)$}
The goal of this section to prove the first main result.
Now, consider our Poisson algebra $R=\SSS(L)$.
We shift the enumeration of the lower central series terms, namely we consider:
\begin{align*}
L=L_0\supseteq L_1\supseteq\cdots\supseteq L_n\supseteq\cdots,\quad  \text{where}\quad L_n&=\gamma_{n+1}(L),\ n\geq0.
\end{align*}
We show that this {\it shifted filtration} $\{L_n\, |\, n\ge 0\}$ induces
the filtration of $\SSS(L)$ by the upper Lie powers $\{\SSS(L)^{(n)}\,|\, n\ge 0\}$, see Lemma~\ref{RiRj=Ri+j}.

\begin{Lemma}\label{Lgamma}
Let $L$ be a Lie algebra over a field of positive characteristic $p$.
\begin{enumerate}
\item
the upper Lie powers  of the truncated symmetric algebra $\SSS(L)$ are as follows
\begin{equation}
\SSS(L)^{(n)}= \sum_{\substack{(m_1-1)+\cdots+{(m_s-1)}\ge n\\
                  1\le m_1\le m_2\le \cdots\le m_s}}
            \gamma_{m_1}(L)\gamma_{m_2}(L)\cdots \gamma_{m_s}(L),\quad n\ge 0;\nonumber
\end{equation}
\item $\{\SSS(L)^{(n)}\,|\, n\geq0\}$ is induced by the shifted filtration
$\{L_i=\gamma_{i+1}(L)\, |\, i\geq0\}$;
\item
$L\cap\SSS(L)^{(n)}=\gamma_{n+1}(L)$, $n\ge 0$ (these are {\em truncated Poisson dimension subalgebras}).
\end{enumerate}
\end{Lemma}

\begin{proof} Consider a finite number of integers $m_j\geq1$ such that $\sum_j(m_j-1)\geq n$.
By~\eqref{gamma(R)em R(n-1)},
we have $\gamma_{m}(L)\subseteq\gamma_{m}(\SSS(L))\subseteq \SSS(L)^{(m-1)}$, $m\geq 1$.
We use Lemma~\ref{RiRj=Ri+j}
$$\gamma_{m_1}(L)\cdots \gamma_{m_s}(L)
  \subseteq \SSS(L)^{(m_1-1)} \cdots \SSS(L)^{(m_s-1)}
  \subseteq \SSS(L)^{\left((m_1-1)+\cdots +(m_s-1)\right)}
  \subseteq \SSS(L)^{(n)},
$$
implying one inclusion of the first claim.
We prove the inverse inclusion by induction on~$n$. The case $n=0$ is trivial.
Let $a\in \gamma_{m_i}(L)$, $m_i\ge 1$, and $w=b_1\cdots b_t\in \SSS(L)$, $b_j\in L$, $j=1,\dots,t$, then
\begin{equation}\label{gamma_n_a}
\{a,w\}=\sum_{j=1}^t b_1\cdots\widehat{b_j}\cdots b_t \{a,b_j\},\qquad
 \{a,b_j\}=[a,b_j]\in \gamma_{m_i+1}(L),\quad  b_k\in L=\gamma_1(L),\ k\ne j.
 \end{equation}
Consider $n\ge 1$. We apply~\eqref{gamma_n_a} to prove the inductive step
\begin{align*}
\SSS(L)^{(n)}&= \{\SSS(L)^{(n-1)},\SSS(L)\}\SSS(L)\\
 &\subseteq \bigg\{ \sum_{\sum_i(m_i-1)\ge n-1} \gamma_{m_1}(L)\cdots \gamma_{m_s}(L),\SSS(L)\bigg\}\SSS(L)
 \subseteq  \sum_{\sum_j(n_j-1)\ge n} \gamma_{n_1}(L)\cdots \gamma_{n_r}(L).
\end{align*}

Now Claim 2 follows by Claim 1 and definitions of the shifted filtration and the induced filtration.
Claim 3 follows by Claim 4 of Lemma~\ref{Filt}.
\end{proof}

\begin{Corr}\label{C_dim_subalg}
Let $L$ be a Lie algebra. Consider the upper Lie powers of the symmetric algebra $S(L)^{(n)}$, $n\ge 0$.
Define {\em Poisson dimension subalgebras}:
$D^{S}_{(n)}(L)=L\cap (S (L))^{(n)}$, $n\ge 0$.
Then $D^{S}_{(n)}(L)=\gamma_{n+1}(L)$, $n\ge 0$.
\end{Corr}
\begin{proof} Follows by the same arguments.
\end{proof}

\begin{Corr} $\SSS(L)^{(n)}$ modulo $\SSS(L)^{(n+1)}$ has the following basis for all $n\geq 0$:
\begin{equation}\label{monom_s}
\Big\langle
\eta=e_{01}^{\alpha_{01}}\cdots e_{0m_0}^{\alpha_{0m_0}}
     \cdots
     e_{s1}^{\alpha_{s1}}\cdots e_{sm_s}^{\alpha_{sm_s}}
\,\Big|\,
\nu(\eta) =\sum_{i=0}^s\sum_{j=1}^{m_i}
                 i\alpha_{ij} = n;\
0 \leq \alpha_{ij} \leq p-1,\ s\geq0
\Big\rangle_K.
\end{equation}
\end{Corr}
\begin{proof}
Follows by Claim 2 of Lemma~\ref{Filt}.
\end{proof}

Let $\{e_{ij}\,|\,j\in J_i, i\geq0\}$ be an ordered basis of $L$ chosen so that $\{e_{ij}\,|\, j\in J_i\}$
is a base for $L_i=\gamma_{i+1}(L)$ modulo $L_{i+1}=\gamma_{i+2}(L)$, for all $i\geq 0$.
We have the height function $\nu(e_{ij})=i$,  $j\in J_i$, $i\geq0$.

\begin{proof} {\it of Theorem~\ref{Tnilp}.}
Implication 1)  $\Rightarrow$ 2) follows by \eqref{gamma(R)em R(n-1)}.
\medskip

Implication 3) $\Rightarrow$ 1).
Suppose that $L$ is nilpotent and $\dim L^2<\infty$.
Then $d_i=\dim\gamma_{i+1}(L)/\gamma_{i+2}(L)<\infty$ for $i\ge 1$.
Let $\{e_{i1},e_{i2},\ldots,e_{id_i}\}$ be a base for
$\gamma_{i+1}(L)$ modulo $\gamma_{i+2}(L)$, for all $i\geq 1$.
Consider a basis element $\eta$ of the truncated symmetric algebra~\eqref{monom_s}.
Put $N=(p-1)\sum_{i\geq1}id_i$.
We have
$\nu(\eta)=\sum_{i=1}^s\sum_{j=1}^{d_i} i\alpha_{ij}\leq (p-1)\sum_{i\geq1}id_i=N.$
Hence, $\SSS(L)^{(N+1)}=0$.
On the other hand, consider
$v=e_{11}^{p-1}\cdots e_{sm_s}^{p-1}$, with $\nu(v)=(p-1)\sum_{i\geq1}id_i=N$, so
$0\ne v\in \SSS(L)^{(N)}$, thus $\SSS(L)^{(N)}\ne 0$.
We proved that $\SSS(L)$ is strongly Lie nilpotent of class $N+1$.
\medskip

Finally, let us prove implication 2) $\Rightarrow$ 3).
Since $L$ is clearly nilpotent, it remains to prove that $L^2$ is finite dimensional.
By Theorem~\ref{Treduction}, there exist integers $n, M$ such that $\dim L/\Delta<n$,
$\Delta=\Delta(L)=\Delta_M(L)$, and $\dim \Delta^2\le M^2$.
It suffices to prove that $L=\Delta$. Let us assume the opposite, that $L\neq\Delta$.

Consider the factor algebra $\bar{L}=L/\Delta^2$.
Since $\dim \Delta^2<\infty$, by Lemma \ref{DeltaI}, $\Delta(\bar{L})=\Delta/\Delta^2$.
Hence, $\Delta(\bar{L})$ is abelian and $\bar{L}\neq\Delta(\bar{L})$.
Replacing $L$ by $\bar{L}$ we may assume that $\Delta$ is abelian.

Fix $x\in L\setminus \Delta$ and consider $H=\langle x\rangle_K\oplus\Delta$.
By Lemma~\ref{DeltaI} we have  $\Delta(H)=\Delta$.
By nilpotency of $\SSS(L)$ there exists an integer $k$ such that $(\ad x)^k(H)=0$.
Let us reduce our arguments to the case $(\ad x)^2(H)=0$.
Put
$D_0=\Delta$ and $D_{m+1}=[x,D_{m}]$, $m\geq 0$. We obtain a chain of $H$-ideals:
\begin{equation}\label{chain}
\Delta=D_0\supseteq D_1=[x, \Delta]=H^2\supseteq D_2\supseteq  \cdots\supseteq D_k=0.
\end{equation}

Since $\dim D_1=\infty$, there exists $i$ (where $1\leq i<k$) such that $\dim D_i/D_{i+1}=\infty$.
Consider $\bar {H}=(\langle x\rangle_K\oplus D_{i-1})/D_{i+1}$,
let $\bar {x}$ be the image of $x$ in $\bar{H}$.
We have
$\dim[\bar{x},\bar{H}]=\dim([x,D_{i-1}]+D_{i+1})/D_{i+1}=\dim D_{i}/D_{i+1}=\infty$.
By construction, $(\ad \bar{x})^2(\bar{H})=0$.

Now our arguments are reduced to the case
$H=\langle x\rangle_K\oplus \Delta$, where $\Delta$ is an abelian ideal, $(\ad x)^2\Delta=0$, and $\dim[x,\Delta]=\infty$.
We find elements $\{y_i \,|\,  i\in\N \}\subseteq \Delta$
such that the elements $z_i= [x,y_i]$, where  $i\in\N$, are linearly independent.
One checks that $\{y_i,z_i\, |\, i\in\N\}$ is a linearly independent set.
We claim that $\{x,xy_1,xy_2,\ldots,xy_n\}=xz_1z_2\cdots z_n$ for all $n\geq1$.
The case $i=1$ follows by the Leibnitz rule. Suppose that the statement is valid for $n-1$.
Observe that $z_1,z_2,\ldots$ are central and we get
\begin{align*}
\{x,xy_1,xy_2,\ldots,xy_n\}
&=\{\{x,xy_1,xy_2,\ldots,xy_{n-1}\},xy_n\}
 =\{xz_1\cdots z_{n-1} ,xy_n\}\\
&=xz_1\cdots z_{n-1}\{x,y_n\}
 =xz_1\cdots z_{n-1}z_n,\quad n\ge 1.
\end{align*}
The right hand side is non-zero by PBW-theorem for all $n\ge 1$,
a contradiction with our assumption that $\SSS(H)$ is Lie nilpotent.
The contradiction proves that $L=\Delta$.
\end{proof}
Above we also proved the following part of Theorem~\ref{Tclasses}.
\begin{Corr}\label{classSpNS}
Let $L$ be a Lie algebra over a field of positive characteristic $p$.
Assume that $L$ is nilpotent and $\dim L^2<\infty$.
Then $\SSS(L)$ is strongly Lie nilpotent of class $$1+(p-1)\sum_{n\geq1}n\cdot\dim(\gamma_{n+1}(L)/\gamma_{n+2}(L)).$$
\end{Corr}

\section{Products of commutators in Poisson algebras}\label{Scomm}
The goal of this section is to prove a technical result on products of
commutators in Poisson algebras  (Theorem~\ref{Tcomutt}) that is used to get a lower bound on
the Lie nilpotency class of $\SSS(L)$.

Products of terms of the lower central series for {\it associative algebras} appear in works of many mathematicians,
the results being reproved without knowing the earlier works.
We do not pretend to make a complete survey here.
Probably, the first observations on products of commutators in associative algebras were
made by Latyshev in 1965~\cite{Lat65} and Volichenko in 1978~\cite{Vol78}.
There are further works e.g.~\cite{BapJor13,EtiKimMa09,Gor07,GriPche15,Kras13}.

In case of associative algebras,
Claim~1 of Theorem~\ref{Tcomutt}, probably, first was established by
Sharma-Shrivastava in 1990,~\cite[Theorem~2.8]{ShaSri90}.
As was remarked in~\cite{RiSh95},
the proof of the associative version of Claim~2 of Theorem~\ref{Tcomutt} is implicitly contained in~\cite{ShaSri90},
where it is proved for group rings.
A weaker statement (the associative version of~Lemma~\ref{Lgamma-2}) is established by
Gupta and Levin~\cite[Theorem~3.2]{GupLev83}.

The following statement is a Poisson version of the respective results for associative algebras.
The validity of it is not automatically clear.
We follow a neat approach due to Krasilnikov~\cite{Kras13}.

\begin{Th}\label{Tcomutt}
Let $R$ be an arbitrary Poisson algebra over a field $K$, $\ch K\neq 2,3$.
\begin{enumerate}
\item suppose that one of integers $n,m\geq 1$ is odd, then
     $$\gamma_n(R)\cdot\gamma_m(R)\subseteq\gamma_{n+m-1}(R) R;$$
\item for all $x_1,\ldots,x_n\in R,$ $n,m\geq 1$ we have
     $$\{x_1,\ldots,x_n\}^m\in\gamma_{(n-1)m+1}(R) R.$$
\end{enumerate}
\end{Th}

We proceed by steps.
\begin{Lemma}\label{Lperm}
Let $R$ be an arbitrary Poisson algebra.
\begin{enumerate}
\item for any $a_1,a_2,a_3,a_4,a_5\in R$, where one of $a_1,a_2,a_5$ belongs to $\gamma_m(R)$, $m\ge 1$, we have
$$ \{a_1,a_2,a_3\}\{a_4,a_5\}+ \{a_1,a_2,a_4\}\{a_3,a_5\}\in\gamma_{m+3}(R)R;$$
\item for any $a_5\in \gamma_m(R)$, $m\ge 1$, and $a_1,a_2,a_4,a_5\in R$ we have
$$ \{a_1,a_2,a_3\}\{a_4,a_5\}+ \{a_1,a_4,a_3\}\{a_2,a_5\}\in\gamma_{m+3}(R)R.$$
\end{enumerate}
\end{Lemma}
\begin{proof} By the Leibnitz rule,
\begin{multline*}
\{a_1,a_2,a_3a_4,a_5 \}=\Big\{a_3\{a_1,a_2,a_4\} + \{a_1,a_2,a_3\}a_4, a_5\Big\}\\
 =a_3\{a_1,a_2,a_4,a_5 \}+\{a_3,a_5\}\{a_1,a_2,a_4\}
  +\{a_1,a_2,a_3\}\{a_4, a_5\}+\{a_1,a_2,a_3,a_5\}a_4.
\end{multline*}
Consider the original product, the first, and the last term obtained.
By properties of Lie commutators,
$\{a_1,a_2,a_3a_4,a_5 \}\in \gamma_{m+3}(R)$,
$\{a_1,a_2,a_4,a_5 \}\in \gamma_{m+3}(R)$,
$\{a_1,a_2,a_3,a_5\}\in \gamma_{m+3}(R)$.
Two remaining middle terms exactly give the first claim
$$ \{a_1,a_2,a_3\}\{a_4,a_5\}+ \{a_1,a_2,a_4\}\{a_3,a_5\}\in\gamma_{m+3}(R)R.$$
Consider the second claim. By the Leibnitz rule:
\begin{gather}
\begin{align*}
&\{a_5,a_2a_4,a_1,a_3 \}=\Big\{a_2\{a_5,a_4\} + \{a_5,a_2\}a_4, a_1,a_3\Big\}\\
&\quad=\Big\{a_2\{a_5,a_4,a_1\}+\{a_2,a_1\}\{a_5,a_4\}+\{a_5,a_2\}\{a_4,a_1\}+ \{a_5,a_2,a_1\}a_4,a_3\Big\}
\end{align*}\\
\begin{align*}
  &=a_2\{a_5,a_4,a_1,a_3\}\\
  &\quad+\{a_2,a_3\}\{a_5,a_4,a_1\}+\{a_2,a_1\}\{a_5,a_4,a_3\}\\
  &\quad+\{a_2,a_1,a_3\}\{a_5,a_4\}+\{a_5,a_2\}\{a_4,a_1,a_3\}\\
  &\quad+\{a_5,a_2,a_3\}\{a_4,a_1\}+\{a_5,a_2,a_1\}\{a_4,a_3\}\\
  &\quad+\{a_5,a_2,a_1,a_3\}a_4.
\end{align*}
\end{gather}
The initial product, the first and last terms belong to $\gamma_{m+3}(R)R$.
We apply the first claim to the sums in the second and forth lines. Each of them belongs to $\gamma_{m+3}(R)R$,
because we permute $a_3, a_1$ in both cases.
Finally, the sum of two elements in the third line yields:
\begin{multline*}
\{a_2,a_1,a_3\}\{a_5,a_4\}+\{a_5,a_2\}\{a_4,a_1,a_3\}\\
=\{a_1,a_2,a_3\}\{a_4,a_5\}+ \{a_1,a_4,a_3\}\{a_2,a_5\}\in\gamma_{m+3}(R)R.
\qedhere
\end{multline*}
\end{proof}

\begin{Lemma}\label{Lprod3}
Let $R$ be a Poisson algebra.
For any $a_5\in \gamma_m(R)$, $m\ge 1$, and $a_1,a_2,a_3,a_4\in R$ we have
$$ 3\{a_1,a_2,a_3\}\{a_4,a_5\}\in \gamma_{m+3}(R)R.$$
\end{Lemma}
\begin{proof} Consider $W=\{a_1,a_2,a_3\}\{a_4,a_5\}$.
Modulo $\gamma_{m+3}(R)R$,
it is skewsymmetric in $a_3,a_4$ (Lemma~\ref{Lperm}, first claim)
and skewsymmetric in $a_2,a_4$ (Lemma~\ref{Lperm}, second claim).
Therefore, it is skewsymmetric in $a_2,a_3$ modulo $\gamma_{m+3}(R)R$.
The skewsymmetry in $a_1,a_2$ follows by anticommutativity.
Therefore, $W$ is skewsymmetric in $a_1,a_2,a_3$. Consider even permutations
\begin{multline*}
\{a_1,a_2,a_3\}\{a_4,a_5\}+\{a_2,a_3,a_1\}\{a_4,a_5\}+\{a_3,a_1,a_2\}\{a_4,a_5\}\\
= 3 \{a_1,a_2,a_3\}\{a_4,a_5\}\mod \gamma_{m+3}(R)R.
\end{multline*}
On the other hand, the left hand side is equal to zero by the Jacobi identity.
The result follows.
\end{proof}

\begin{Lemma}\label{Lprod2}
Let $R$ be a Poisson algebra over a field $K$, $\ch K\neq 2,3$. Then
$$\{\gamma_m(R)R, R, R\}\subseteq \gamma_{m+2}(R)R,\quad m\ge 2.$$
\end{Lemma}
\begin{proof} Let $c\in\gamma_{m}(R)$ and $x,y,z\in R$. Consider the product
\begin{equation}\label{cxyz}
\{cx,y,z \}=c\{x,y,z\}+\{c,y\}\{x,z\}+\{c,z\}\{x,y\}+ \{c,y,z\}x.
\end{equation}
By properties of Lie commutators, $\{c,y,z\}x\in \gamma_{m+2}(R)R$.

Since $c\in\gamma_{m}(R)$, $m\ge 2$,
we can consider that $c=\{a,b\}$, where $a\in R$, $b\in\gamma_{m-1}(R)$.
Applying Lemma~\ref{Lprod3}, we get $c\{x,y,z\}\in\gamma_{m+2}(R)R$.

Consider  two remaining middle terms in~\eqref{cxyz}
\begin{align*}
\{c,y\}\{x,z\}+\{c,z\}\{x,y\}
&= \{z,x\}\{y,c\}+\{y,x\}\{z,c\}\\
&=\{ zy, x,c\}-z\{ y, x,c\}-\{z, x,c\}y\in \gamma_{m+2}(R)R.
\end{align*}
Thus, $\{cx,y,z \}\in \gamma_{m+2}(R)R$, yielding the result.
\end{proof}
Now we prove the first claim of Theorem~\ref{Tcomutt} as a separate statement.
\begin{Lemma}
Let $R$ be a Poisson algebra over a field $K$, $\ch K\neq 2,3$.
Suppose that one of integers $n,m\geq 1$ is odd, then
$\gamma_n(R) \gamma_m(R)\subseteq\gamma_{n+m-1}(R) R$.
\end{Lemma}
\begin{proof} We assume that $n$ is odd and proceed by induction on $n$.
The case $n=1$ is trivial.
The base of induction $n=3$ is proved in Lemma~\ref{Lprod3}.

Let $n\ge 5$ be odd.
Take $a\in\gamma_{n}(R)$, $b\in \gamma_m(R)$.
Since Lie products are linear combinations of left-normed Lie products,
without loss of generality assume that $a=\{c,x,y\}$, $c\in \gamma_{n-2}(R)$, $x,y\in R$.
Then
\begin{equation} \label{cbxy}
a\cdot b=\{c,x,y\}b  
=\{cb,x,y\}- \{c,x\}\{b,y\}- \{c,y\}\{b,x\}- c\{b,x,y\}.
\end{equation}

Consider the first term in~\eqref{cbxy}.
By the inductive assumption, $cb\in\gamma_{n+m-3}(R)R$, and by Lemma~\ref{Lprod2},
we get $\{cb,x,y\}\in\gamma_{n+m-1}(R)R$.

Consider the last term in~\eqref{cbxy}. Clearly, $\{b,x,y\}\in\gamma_{m+2}(R)$, recall that $c\in\gamma_{n-2}(R)$ and $n-2$ is odd.
By the inductive assumption, $c\{b,x,y\}\in \gamma_{n+m-1}(R)R$.

Consider two remaining middle terms in~\eqref{cbxy}.
\begin{align*}
 \{c,x\}\{b,y\}+\{c,y\}\{b,x\}
 =\{y,b\}\{x,c\}+\{x,b\}\{y,c\}
 =\{yx,b,c\}- \{y,b,c\}x - y\{x,b,c\}.
\end{align*}
In the Lie brackets above, $c\in \gamma_{n-2}(R)$, $b\in\gamma_m(R)$, and the remaining element belongs to $R$.
Thus, all three products above belong to $\gamma_{n+m-1}(R)R$.
Lemma is proved.
\end{proof}

We prove the second claim of Theorem~\ref{Tcomutt} as a separate statement.
\begin{Lemma}
Let $R$ be a Poisson algebra over a field $K$, $\ch K\neq 2,3$.
Let $x_1,\ldots,x_n\in R$, $n,m\ge 1$.  Then
$\{x_1,\ldots,x_n\}^m\in\gamma_{(n-1)m+1}(R) R$.
\end{Lemma}
\begin{proof}
Case 1: $n$ is odd. The statement follows by consecutive application of Claim 1 of Theorem~\ref{Tcomutt}.

Case 2: $n$ is even and $m=2$. Denote $a=\{x_1,\ldots,x_{n-1}\}$, $x=x_n$.
Then $a\in \gamma_{n-1}(R)$ and $\{x_1,\ldots,x_n\}=\{a,x\}$. Consider
\begin{align*}
\Big\{\{x^2,a\},a\Big\} &=\Big\{2x\{x,a\},  a\Big\}
  =2\{x,a\}\{x,a\}+2x \{x,a,a\};\\
\{x,a\}^2  &=\frac 12 \Big\{\{x^2,a\},a\Big\}-x \{x,a,a\}\in \gamma_{2n-1}(R)R.
\end{align*}

Case 3: $n$ is even and $m=2q$, $q>1$. We apply the previous cases.
\begin{align*}
\{x_1,\ldots,x_n\}^{2q}&=(\{x_1,\ldots,x_n\}^2)^q\in (\gamma_{2n-1}(R)R)^q\\
&= (\gamma_{2n-1}(R))^q R\subseteq \gamma_{(2n-2)q+1}(R)R=\gamma_{(n-1)m+1}(R)R.
\end{align*}

Case 4: $n$ is even and $m=2q+1$, $q\ge 1$ (the case $m=1$ is trivial). We use the previous cases.
\begin{align*}
\{x_1,\ldots,x_n\}^{m}&=\{x_1,\ldots,x_n\}^{2q}\{x_1,\ldots,x_n\}\\
&\in \gamma_{(n-1)2q+1}(R)\gamma_{n}(R)\subseteq \gamma_{(n-1)(m-1)+n}(R))R=\gamma_{(n-1)m+1}(R)R.
\qedhere
\end{align*}
\end{proof}

The following is an analogue of the respective fact for associative algebras, see~\cite[Theorem~3.2]{GupLev83}.
It is weaker than Claim~1 of Theorem~\ref{Tcomutt}, but it is valid for an arbitrary characteristic.
\begin{Lemma}\label{Lgamma-2}
Let $R$ be a Poisson algebra over an arbitrary field $K$. Then
$$
\gamma_m(R)\gamma_n(R)\subseteq \gamma_{m+n-2}(R)R, \qquad n,m\ge 2.
$$
\end{Lemma}
\begin{proof}
We proceed by induction on $n\ge 2$ for an arbitrary $m\ge 2$.
In case $n=2$ there is nothing to prove.
Assume that $n> 2$.
To study $\gamma_m(R)\gamma_n(R)$ it is sufficient to consider products of the form
$u=\{r,a\}\{x,y,z_1,\dots,z_{n-2}\}$, where $r\in\gamma_{m-1}(R)$ and all other elements belong to $R$.
Consider $w=\{a\{x,y\},r,z_1,\ldots,z_{n-2}\}$, we get $w\in\gamma_{n+m-2}(R)$.
On the other hand
\begin{align*}
w=\Big\{a\{x,y\},r,z_1,\ldots,z_{n-2}\Big\}
 &=\{a,r\}\{x,y,z_1,\dots,z_{n-2}\}\\
 &\quad +\sum_{s=1}^{n-2}\sum_{i,j}\{a,r,z_{i_1},\ldots,z_{i_s}\}\{x,y,z_{j_1},\dots,z_{j_{n-2-s}}\}\\
 &\quad +\sum_{s=0}^{n-2}\sum_{i,j}\{a,z_{i_1},\ldots,z_{i_s}\}\{x,y,r,z_{j_1},\dots,z_{j_{n-2-s}}\},
\end{align*}
where the sums above correspond to all partitions of the set $\{1,\ldots,n-2\}$ in two subsets of indices
such that $i_1<\cdots< i_s$, $j_1<\cdots< j_{n-2-s}$.
The first term yields the desired product $-u$.
Products of the first sums belong to $\gamma_{m+s}(R)\gamma_{n-s}(R)\in \gamma_{n+m-2}(R)R$
by the inductive assumption because $s>1$.
Products of the second sums belong to $\gamma_{s+1}(R)\gamma_{n+m-s-1}(R)\in \gamma_{n+m-2}(R)R$
by the inductive assumption because $s+1<n$.
Lemma is proved.
\end{proof}
\section{Lie nilpotency classes of truncated symmetric algebras $\SSS(L)$}

Recall that the strong Lie nilpotency class is an upper bound for the ordinary Lie nilpotency class.
Actually, below we prove that these classes coincide for many examples of Poisson algebras that we are studying.
We essentially apply the results of the previous section on products of commutators in Poisson algebras.
The following result and Corollary~\ref{classSpNS} yield Theorem~\ref{Tclasses}.
\begin{Th}\label{Sharma1}
Let $L$ be a Lie algebra over a field of positive characteristic $p>3$.
Assume that $L$ is nilpotent and $\dim L^2<\infty$.
The Lie nilpotency class of $\SSS(L)$ coincides with the strong Lie nilpotency class and is equal to
$$1+(p-1)\sum_{n\geq1}n\cdot\dim(\gamma_{n+1}(L)/\gamma_{n+2}(L)).$$
\end{Th}
\begin{proof}
Corollary~\ref{classSpNS} yields the required value for the strong Lie nilpotency class.
We have finite numbers $d_i=\dim \gamma_{i+1}(L)/\gamma_{i+2}(L)$ for $i\geq 1$.
Let $d_s\neq 0$ but $d_{s+1}=d_{s+2}=\ldots =0$, for some $s$.
Let $\{e_{ij}\, |\, 1\leq j\leq d_i\}$ be a base for $\gamma_{i+1}(L)$ modulo $\gamma_{i+2}(L)$, where $1\leq i\leq s$.
Consider
\begin{equation}\label{v}
0\neq v=e_{11}^{p-1}\cdots e_{1d_1}^{p-1}
 \cdots e_{s1}^{p-1}\cdots e_{sd_s}^{p-1}.
\end{equation}
Observe that the basis elements above can be chosen as commutators,
namely $e_{ij}=\{y_1,\ldots,y_{i+1}\}$,
where $y_k\in L$.
By Claim~2  of Theorem~\ref{Tcomutt},
$e_{ij}^{p-1}\in\gamma_{i(p-1)+1}(R) R$, for all $j=1,\dots, d_i$, $i=1,\dots,s$.
Since $i(p-1)+1$ is odd, we can apply Claim~1 of Theorem~\ref{Tcomutt} and obtain that
$e_{i1}^{p-1}\cdots e_{id_i}^{p-1}\in\gamma_{d_i i(p-1)+1}(R) R$.
Similarly, the total product~\eqref{v} enjoys the property
$v\in\gamma_{N+1}(R) R$, where $N=(p-1)\sum_{i=1}^s id_i$.
Therefore, $\gamma_{N+1}(R)\ne 0$.
On the other hand, $\gamma_{N+2}(R)\subseteq R^{(N+1)}=0$
by~\eqref{gamma(R)em R(n-1)} and Corollary~\ref{classSpNS}.
We proved that $\SSS(L)$ is Lie nilpotent of the required class $N+1$.
\end{proof}

Now, let us establish bounds for the ordinary Lie nilpotency class in case $\ch K=2,3$.
\begin{Lemma}\label{L_cl_lower}
Let $L$ be a non-abelian Lie algebra over a field of characteristic $p=2,3$.
Assume that $L$ is nilpotent and $\dim L^2<\infty$.
The Lie nilpotency class of $\SSS(L)$ is bounded from above by the strong Lie nilpotency class
and bounded from below by the following number
$$2+(p-1)\sum_{n\geq2}(n-1)\cdot\dim(\gamma_{n+1}(L)/\gamma_{n+2}(L)).$$
\end{Lemma}
\begin{proof}
Let bases elements $e_{ij}\in\gamma_{i+1}(L)$ be chosen as in the proof above.
Applying Lemma~\ref{Lgamma-2}, we get
$e_{ij}^{p-1}\in\gamma_{(p-1)(i-1)+2}(R) R$, where  $1\le j\le d_i$, $1\le i\le s$.
(Since $L$ is non-abelian, there exists at least one  basis element
$e_{11}\in\gamma_2(L)$, and product~\eqref{v} is non-trivial.)
Applying Lemma~\ref{Lgamma-2} to the whole of product~\eqref{v}, we get
$0\neq v\in\gamma_{M}(R) R$, where $M=(p-1)\sum_{i=1}^s (i-1)d_i+2$.
Thus, $\gamma_{M}(R)\ne 0$.
Therefore, $M$ is a lower bound for the Lie nilpotency class of $\SSS(L)$.
\end{proof}

\section{Solvability of truncated symmetric algebras $\SSS(L)$}

In this section we prove \textit{Theorem \ref{Tsolv} in case $p\ne 2$}.
\medskip

Implication 1) $\Rightarrow$ 2) follows by inclusion~\eqref{deltindeltb}.
\medskip

Now, we prove implication 3) $\Rightarrow$ 1).
Suppose that $\dim L^2<\infty$ and $L$ is solvable of length $k\ge 1$.
We proceed by induction on $k$.
The base of induction $k=1$ is trivial  because $L$ is abelian and $\SSS(L)$ is abelian as well.

Assume that $k>1$. We have $\delta_{k}(L)=0$ and $\delta_{k-1}(L)\ne 0$.
Put $\bar{L}=L/\delta_{k-1}(L)$ and consider the induced map
$\varphi: \SSS(L)\rightarrow \SSS(\bar{L})$.
By the induction hypothesis, $\SSS(\bar{L})$ is strongly solvable,
then there exists $m$ such that ${\tilde\delta}_m(\SSS(\bar L))=0$,
thus ${\tilde \delta}_m(\SSS(L))\subseteq J=\SSS(L)\delta_{k-1}(L)$, the latter being the kernel of $\varphi$.

Let $\{b_1,\ldots,b_l\}$ be a basis of $\delta_{k-1}(L)$ and
$\{a_i\in L \mid i\in I\}$ a basis of $L$ modulo $\delta_{k-1}(L)$.
Consider a PBW-basis monomial $w\in J$, it has a form $w=a^{\alpha}b^{\beta}$,
where $a^{\alpha}=a_1^{\alpha_1}a_2^{\alpha_2}\cdots a_n^{\alpha_n}$, $0\leq \alpha_i<p$ and
$b^\beta=b_1^{\beta_1}b_2^{\beta_2}\cdots b_l^{\beta_l}\in \SSS(\delta_{k-1}(L))$,
$0\leq\beta_j\leq p-1$, where at least one $\beta_j$ is nonzero.
Denote $|\beta|=\beta_1+\cdots+\beta_l$. We have $|\beta|\ge 1$.

By definition of the upper derived series,
elements of $\tilde{\delta}_{m+1}(\SSS(L))$ are expressed via $\{w_1,w_2\}v$
where $w_1,w_2\in J$, which we write as above, and $v\in \SSS(L)$.
Compute using the Leibnitz rule $\{w_1,w_2\}v=\sum_j u_j$,
where we get a sum of products containing one factor of type
$\{a_i,a_j\}$ or $\{a_i,b_j\}\in \delta_{k-1}(L)$
(as $\delta_{k-1}(L)$ is an ideal of $L$), recall that $\{b_i,b_j\}=0$ because $\delta_{k-1}(L)$ is abelian.
In all cases we obtain monomials with $|\beta''|\geq|\beta|+|\beta'|\geq 2$.
Thus, the PBW-monomials of $\tilde{\delta}_{m+1}(\SSS(L))$ contain at least two elements $b_j$.
Continuing this process, we conclude that monomials of $\tilde{\delta}_{m+N}(\SSS(L))$ contain at least $2^N$ such elements.
Recall that $\dim \delta_{k-1}(L)=l$ and $\dim \SSS(\delta_{k-1}(L))=p^l$.
As soon as  $2^N>p^l$, we get $\tilde{\delta}_{m+N}(\SSS(L))=0$.
We proved that $\SSS(L)$ is strongly solvable.
\medskip

The rest of the section is devoted to implication 2) $\Rightarrow$ 3).
We assume that $\SSS(L)$ is solvable. Since $L\subseteq \SSS(L)$, clearly $L$ is solvable,
it remains to prove that $\dim L^2<\infty$.
By Theorem~\ref{Treduction}, there exist integers $n, M$ such that $\dim L/\Delta<n$,
where $\Delta=\Delta(L)=\Delta_M(L)$ and $\dim \Delta^2\le M^2$.
So, it suffices to prove that $L=\Delta$. Let us assume the opposite, that $L\neq\Delta$,
in the rest of the section we prove that this assumption leads to a contradiction.
Consider $\bar{L}=L/\Delta^2$.
Since $\dim \Delta^2<\infty$, by Lemma \ref{DeltaI}, $\Delta(\bar{L})=\Delta/\Delta^2$.
Hence, $\Delta(\bar{L})$ is abelian and $\bar{L}\neq\Delta(\bar{L})$.
Replacing $L$ by $\bar{L}$ we may assume that $\Delta$ is abelian.
Since our identity is multilinear, we can assume that $K$ is algebraically closed.

\begin{Lemma}\label{ad}
Suppose that $L$ is a Lie algebra over a field $K$,
$x\in L\setminus \Delta$ and $\Delta$ an abelian ideal.
Let $\SSS(L)$ (or symmetric algebra $S(L)$) be solvable of length $s$.
Then $\ad x$ acts on $\Delta$ algebraically, of degree bounded by $M(s)=(s+1)2^s+s$.
\end{Lemma}

\begin{proof}
By assumption,  $\SSS(L)$ satisfies the identity $\delta_s(x_1,x_2,\ldots,x_{2^s})\equiv0$.
Fix $x\in L\setminus\Delta$. Consider an element $0\neq v\in\Delta$.
Denote $v_i=(\ad x)^i v$ for $i\ge 0$.
We want to prove that $\{v_i \,|\, 0\le i\le M\}$ is a linearly dependent set, where $M=M(s)$.
By way of contradiction, assume that $\{v_i\, |\, 0\le i\le M\}$ is linearly independent.
We shall show that, choosing $N$ sufficiently large, $\delta_s(xv_N,xv_{2N},\ldots,xv_{2^sN})\ne 0$.

We compare elements of the same length $u=xv_{\alpha_1}v_{\alpha_2}\cdots v_{\alpha_k}$,
$w=xv_{\beta_1}v_{\beta_2}\cdots v_{\beta_k}$ lexicographically starting with the senior indices.
Namely, assume that $\alpha_1\leq\alpha_2\leq\cdots \leq\alpha_k$, $\beta_1\leq\beta_2\leq\cdots\leq\beta_k$.
Then $u<w$ if and only if $\alpha_k<\beta_k$, or $\alpha_k=\beta_k$, but $\alpha_{k-1}<\beta_{k-1}$, etc.

Let $k=1$, consider the following evaluation $x_1=xv_{N}, x_2=xv_{2N}$:
\begin{align*}
\delta_1(x_1,x_2)&=\{x_1,x_2\}=\{xv_N,xv_{2N}\}
  =x(v_Nv_{2N+1}-v_{N+1}v_{2N}),
\end{align*}
where the elements $xv_{N}v_{2N+1}$ and $xv_{N+1}v_{2N}$ are nonzero and different.

Let $k=2$, consider the evaluation by $x_1=xv_{N},x_2=xv_{2N},x_3=xv_{3N}$, $x_4=xv_{4N}$:
\begin{align*}
\delta_2(xv_{N},xv_{2N},xv_{3N},xv_{4N})&=\{\delta_1(xv_N,xv_{2N}),\delta_1(xv_{3N},xv_{4N})\}\\
&=\{xv_Nv_{2N+1}-xv_{N+1}v_{2N},xv_{3N}v_{4N+1}-xv_{3N+1}v_{4N}\}\\
&=x(v_Nv_{2N+1}-v_{N+1}v_{2N})(v_{3N}v_{4N+2}-v_{3N+2}v_{4N})\\
&\quad +x(v_{N+2}v_{2N}-v_{N}v_{2N+2})(v_{3N}v_{4N+1}-v_{3N+1}v_{4N}),
\end{align*}
where the element $xv_{N}v_{2N+1}v_{3N}v_{4N+2}$ is the unique maximal with respect  to the lexicographic order.

In general, let $k\geq 1$, making a similar evaluation $x_i=xv_{iN}$, for all $i=1,2,\ldots, 2^k$,
we present $\delta_k(x_1,x_2,\ldots,x_{2^k})$ as a sum
\begin{align}\label{epsilon}
\sum_{\epsilon_j}\theta_{\epsilon} xv_{N+\varepsilon_1}v_{2N+\varepsilon_2}\cdots v_{2^kN+\varepsilon_{2^k}},\quad
\theta_{\epsilon}\in K;\quad
\text{where $\epsilon_j$ satisfy:}\\
\epsilon_1+\epsilon_2+\cdots+\epsilon_{2^k}=2^k-1,\qquad
 0\leq\varepsilon_i\leq k, \quad i=1,2,\ldots, 2^k.
 \nonumber
\end{align}
The sum contains a unique largest product having $\theta_{\epsilon}=\pm 1$.
Indeed, let for $k$ we have a unique maximal element (\ref{epsilon}) given by a tuple $(\alpha_1,\alpha_2,\ldots,\alpha_{2^k})$,
then the unique maximal element for $k+1$ is given by the tuple
$(\alpha_1,\ldots,\alpha_{2^k},\alpha_1,\ldots,\alpha_{2^k-1},\alpha_{2^k}+1)$.
Also, we check by induction that $\epsilon_i\le k$.

Now, we consider $k=s$ and fix $N=s+1$ to guarantee that the indices of $v_i$s in (\ref{epsilon}) do not overlap.
Observe that the highest index in \eqref{epsilon} is bounded by $M(s)=(s+1)2^s+s$.
The unique maximal tuple yields a non-zero product
$xv_{N+\a_1}v_{2N+\a_2}\cdots v_{2^sN+\a_{2^s}}$ in~\eqref{epsilon}.
A contradiction to the fact that $\delta_s(x_1,x_2,\ldots,x_{2^s})\equiv0$ in $\SSS(L)$.

The contradiction proves that for each $v\in \Delta$
there exists a polynomial $h_v(t)$, such that $h_v(\ad x)(v)=0$ and $\deg h_v(t)\le M$.
Therefore, $\Delta$ is a periodic $K[t]$-module under the action $t\circ w=(\ad x) w$, $w\in \Delta$.
Consider the respective  Jordan decomposition
$$\Delta=\mathop{\oplus}\limits_{i\in I}V_{(\lambda_i)},\qquad
  V_{(\lambda_i)}=\big\{y\in \Delta\,\big\vert\, (\ad x-\lambda_i E)^{m_i}y=0\big\},\ \lambda_i\in K,\ i\in I,$$
where $m_i$ is minimal.
The minimal polynomial of the action of $\ad x$ on $\Delta$ is $h(t)=\prod_{i\in I}(t-\lambda_i)^{m_i}$.
Indeed, it is sufficient to consider the action on elements $v=\sum_{i\in I} v_i$,
where $0\ne v_i\in V_{(\lambda_i)}$ correspond to minimal degrees $m_i$, $i\in I$.
Hence, $\deg h(t)=\sum_{i\in I} m_i\le M(s)$.
Thus, $\ad x$ acts algebraically on $\Delta$ of degree bounded by $M(s)$.
\end{proof}
\begin{Corr}\label{Ceigen}
Let $x\in L\setminus \Delta$ act on $\Delta$ as above. Then
\begin{enumerate}
\item The number of different eigenvalues $\lambda_i$, $i\in I$, is bounded by M(s).
\item Let $V_{\lambda_i}$, $V_{(\lambda_i)}$ be the eigenspace and respective Jordan subspace
corresponding to an eigenvalue~$\lambda_i$. Then
 $\dim V_{(\lambda_i)}\leq \dim V_{\lambda_i}\cdot M(s)$, $i\in I$.
\item There exists an eigenvalue $\lambda$ with $\dim [x,V_{(\lambda)}]=\infty$.
\item Let $\lambda$ be as above, then $\dim V_{\lambda}=\infty$.
\end{enumerate}
\end{Corr}
\begin{proof}
The first claim follows by the bound on the degree of the minimal polynomial $h(t)$ found above.
The second claim follows by the fact that the size of each Jordan matrix block corresponding to a given $\lambda_i$ is bounded by $M(s)$.

Consider $H=\langle x\rangle_K\oplus\Delta$.
By Lemma~\ref{DeltaI},  $\Delta(H)=\Delta$, thus $\dim [x,\Delta]=\infty$.
Since $[x,\Delta]=\sum_{i\in I}[x, V_{(\lambda_i)}]$,
the third claim follows.

Claim 4 follows from Claims~2 and~3.
\end{proof}

Now, by Corollary~\ref{Ceigen}, our proof is reduced to the following situation.
We have $L=\langle x\rangle\oplus V_{(\lambda)}$ with $\dim [x,V_{(\lambda)}]=\infty$ and $\dim V_\lambda=\infty$.
In our proof we have two cases: A) $\lambda=0$, and B) $\lambda\neq 0$.
\medskip

{\bf The case A:} $\lambda=0$. Now $V_{(0)}=\{ y\, |\, (\ad x)^{m}y=0\}$, where $m\ge 2$ is minimal.
We can assume that $(\ad x)^2 V_{(0)}=0$.
This reduction is made as in the proof on Lie nilpotence (see application of the chain~\eqref{chain}).
The case is reduced to the following statement, in which proof we follow approach of~\cite{RiSh93}.

\begin{Pro}\label{SpSp}
Let $L$ be a Lie algebra over a field of odd characteristic such that
\begin{enumerate}
\item $L=\langle x\rangle\oplus\Delta$, where $\Delta$ is an abelian ideal;
\item $(\ad x)^2\Delta=0$, $\dim [x,\Delta]=\infty$.
\end{enumerate}
Then $\SSS(L)$ is not solvable.
\end{Pro}

We adopt the following notations. Denote $R=\SSS(L)$ and $a'=\{x,a\}$, for any $a\in \SSS(\Delta)$.
Denote $\SSS(x)=\langle x^i\,|\, 0\leq i\leq p-1\rangle_K$.
We have $R=\SSS(x) \SSS(\Delta)$.
Let $S=\{a\in \SSS(\Delta)\,|\, a''=0\}$. Note that $S$ is a linear subspace of $\SSS(\Delta)$, but it need not be a subring.
However, we do have the following result.

\begin{Lemma}
Let $a,b \in S$. Then $a'b\in S$ and $a'b+ab'\in S$.
\end{Lemma}\label{a'b}
\begin{proof}
By the Leibnitz rule, $(a'b)''=\{x,\{x,a'b\}\}=\{x,a'b'\}=0$.
\end{proof}
\begin{Lemma}
Let $a,b\in S$. For all $0\le i,j<p$ we have:
\begin{enumerate}
\item $\{x^i,a\}=ix^{i-1}a'$;
\item $\{x^ia,x^jb\}=x^{i+j-1}(iab'-ja'b)$.
\end{enumerate}\label{axibxj}
\end{Lemma}
\begin{proof} By induction.
\end{proof}

\begin{Lemma}\label{i=p-1}
Let $a,b\in S$ such that $\SSS(x) a, \SSS(x) b\subseteq \delta_{n-1}(R)$.
Then $\SSS(x)(a'b+ab')\subseteq\delta_n(R)$.
\end{Lemma}
\begin{proof} By Lemma \ref{axibxj}, $\delta_n(R)\ni\{x^ia,b\}=ix^{i-1}ab'$, $i=1,\ldots,p-1$.
Hence, $x^iab'\in\delta_n(R)$, and similarly $x^ia'b\in\delta_n(R)$, for $i=0,1,\ldots,p-2$.
It remains to settle the case $i=p-1$. By Lemma \ref{axibxj},
$ x^{p-1}(ab'+a'b)=x^{p-1}(ab'-(p-1)a'b) =\{x a,x^{p-1}b\}\in\delta_n(R)$.
\end{proof}

We supply $S$ with a {\it new bilinear operation} $(a,b)\mapsto a\diamond b=a'b+ab'$, $a,b\in S$.
We define a new {\it ``derived series''} of $S$ by $S^{[[0]]}=S$ and
$S^{[[n]]}=S^{[[n-1]]}\diamond S^{[[n-1]]}$,
the linear space generated by the elements $a\diamond b$ where $a,b\in S^{[[n-1]]}$, where $n\geq 1$.

\begin{Lemma}\label{<x>S{{n}}}
$\SSS(x) S^{[[n]]}\subseteq\delta_n(R)$ for all $n\geq 0$.
\end{Lemma}
\begin{proof}
Follows from Lemma \ref{i=p-1} by induction on $n$.
\end{proof}
For all $a_1,a_2,\ldots, a_{l}\in S$, where $l\ge 1$, let us define a {\it new product}:
\begin{equation}\label{a1an}
[[a_1,a_2,\dots,a_l]]=a_1'a_2'\cdots a_{l-1}'a_l+a_1'a_2'\cdots a_{l-2}'a_{l-1}a_l'+\cdots +a_1a_2'\cdots a_l'.
\end{equation}
\begin{Lemma}\label{a{2^n}}
$\big\{[[a_1,a_2,\dots,a_{2^n}]]\, \big|\, a_1,a_2,\dots,a_{2^n}\in S\big\}\subseteq S^{[[n]]}$ for all $n\ge 0$.
\end{Lemma}

\begin{proof}
We proceed by induction on $n$.
Consider $n=0$, we have $[[a_1]]=a_1\in S=S^{[[0]]}$.
Consider $n=1$, we have $[[a_1,a_2]]=a_1'a_2+a_1a_2'=a_1\diamond a_2\in S\diamond S=S^{[[1]]}$.

Assume that the assertion is valid for $n\ge 1$ and consider $n+1$.
Note that, since $a_1,a_2,\dots,a_{2^n}\in S$,  we have
$[[a_1,a_2,\dots,a_{2^n}]]'=2^na_1'a_2'\cdots a_{2^n}'.$
Hence,
\begin{align*}
&[[a_1,a_2,\dots,a_{2^n}]]\diamond[[b_1,b_2,\dots,b_{2^n}]]\\
&\quad=2^n\Big(a_1'a_2'\cdots a_{2^n}'[[b_1,b_2,\dots,b_{2^n}]]+[[a_1,a_2,\dots,a_{2^n}]]b_1'b_2'\cdots b_{2^n}'\Big)\\
&\quad=2^n[[a_1,a_2,\dots,a_{2^n},b_1,b_2,\dots,b_{2^n}]].
\end{align*}
By induction hypothesis, $[[a_1,a_2,\dots,a_{2^n}]],[[b_1,b_2,\dots,b_{2^n}]]\in S^{[[n]]}.$
Since $p\ne 2$, we get
\begin{equation*}
[[a_1,a_2,\dots,a_{2^n},b_1,b_2,\dots,b_{2^n}]]\in S^{[[n]]}\diamond S^{[[n]]}=S^{[[n+1]]}.
\qedhere
\end{equation*}
\end{proof}

\begin{Lemma}\label{Ldelta_n}
For all $a_1,a_2,\dots,a_{2^n}\in S$ we have
$\SSS(x) [[a_1,a_2,\dots,a_{2^n}]]\subseteq \delta_n(R)$, where $n\geq 0$.
\end{Lemma}
\begin{proof}
Follows form Lemma~\ref{<x>S{{n}}} and Lemma~\ref{a{2^n}}.
\end{proof}

\begin{proof} {\it of Proposition~\ref{SpSp}}.
Since $\dim [x,\Delta]=\infty$,
we can choose $\{a_n \,|\, n\in\N \}\subseteq\Delta$ such that $\{a_n' \,|\, n\in\N\}\subseteq L^2$
is a linearly independent set.
One checks that $\{a_n, a_n'\,|\, n\in\N \}$ is linearly independent as well.
Since $\SSS(\Delta)$ is a truncated polynomial ring, monomials in~\eqref{a1an} are
linearly independent.
Therefore, $[[a_1,a_2,\dots,a_{2^n}]]\neq0$
and $\delta_n(R)\neq0$ for $n\ge 0$ by Lemma~\ref{Ldelta_n}.
So, $\SSS(L)$ is not solvable.
\end{proof}

{\bf The case B:} $\lambda\neq0$. By Corollary~\ref{Ceigen}, we have the eigenspace
$V_\lambda$ with $\dim V_\lambda=\infty$.
Consider a subalgebra $\langle x\rangle \oplus V_\lambda\subseteq L$,
by taking $\tilde x=\frac 1\lambda x$, we can assume that $\lambda=1$.
The case is reduced to the following statement.

\begin{Pro}\label{SpSp2}
Let $L$ be a Lie algebra over a field of odd characteristic such that
\begin{enumerate}
\item $L=\langle x\rangle\oplus\Delta$, where $\Delta$ is an abelian ideal;
\item $[x,v]=v$, for all $v\in\Delta$, and $\dim\Delta=\infty$;
\end{enumerate}
Then $\SSS(L)$ is not solvable.
\end{Pro}

\begin{proof}
We choose a linearly independent set $\Xi=\{v_i\,|\, i\in\N\}\subseteq \Delta $.
Next, we construct sets of multilinear monomials in $\Xi$ of type $v=v_{i_1}\cdots v_{i_k}\in \SSS(\Delta)$,
denote their lengths as $|v|=k$.
Observe that $v'=\{x,v\}=|v|v$.
We start with $e_1=v_{i_1}$, $h_1=v_{i_1}v_{i_2}$, $f_1=v_{i_1}v_{i_2}v_{i_3}$, $i_1<i_2<i_3$,
where these symbols denote sets of all multilinear monomials in $\Xi$ of lengths 1,2,3, respectively.
Next, define recursively sets of multilinear monomials in $\Xi$:
\begin{equation*}
e_{n+1}=e_nh_n,\quad
h_{n+1}=e_nf_n,\quad
f_{n+1}=h_nf_n,\quad n\ge 1.
\end{equation*}
where we include these products in our new sets only when they are multilinear.
One checks by induction that
\begin{equation*}
|e_{n}|=2^{n}-1,\quad
|h_{n}|=2^n,\quad
|f_{n}|=2^n+1,\quad n\ge 1.
\end{equation*}
Thus, $e_n$, $h_n$, $f_n$ consist of all multilinear monomials in $\Xi$ of the respective lengths.
Let us prove by induction on $n$ that $xe_n,xh_n,xf_n\in \delta_{n-1}(\SSS(L))$, $n\ge 0$.
The base of induction $n=0$ is trivial.
Assume the claim for $k=n$.
We have the following relations
\begin{align*}
\{xe_n,xh_n \}=x(e_nh_n'-e_n'h_n)=x(|h_n|-|e_n|)e_nh_n=xe_nh_n;\\
\{xh_n,xf_n \}=x(h_nf_n'-h_n'f_n)=x(|f_n|-|h_n|)h_nf_n=xh_nf_n;\\
\{xe_n,xf_n \}=x(e_nf_n'-e_n'f_n)=x(|f_n|-|e_n|)e_nf_n=2xe_nf_n.
\end{align*}
Considering only multilinear products above,
we get the inductive step: $xe_{n+1}, xh_{n+1}, xf_{n+1}\in \delta_n(\SSS(L))$.
Therefore, $\delta_n(\SSS(L))\ne 0$ for all $n\ge 0$.
We proved that $\SSS(L)$ is not solvable.
\end{proof}
\begin{proof}{\it of Theorem~\ref{Tsolv} in case $p\ne 2$}. In this section, the proof was reduced  to
two cases settled in Proposition~\ref{SpSp} and Proposition~\ref{SpSp2}.
They lead to a contradiction, which implies that $L=\Delta$ and $\dim L^2=\dim \Delta^2<\infty$.
\end{proof}

\section{Solvability of truncated symmetric algebras $\SSS(L)$, $\ch K=2$}

In this section we consider the case $\ch K=2$.
We finish the proof of a part of Theorem~\ref{Tsolv} on the strong solvability of $\SSS(L)$ and show
that the question of the solvability of $\SSS(L)$ is different to other characteristics.
Namely, we obtain two examples of the truncated symmetric algebras that are solvable but not strongly solvable.
These examples also show that Proposition~\ref{SpSp} and Proposition~\ref{SpSp2} are no longer valid in case $\ch K=2$.

\begin{Lemma}\label{L22}
Let $L=\langle x, y_i \,|\,  [x,y_i]=y_i,\, i\in\N\rangle_K$, $\ch K=2$,
the remaining commutators being trivial. Then
\begin{enumerate}
\item $L^2=\Delta(L)=\Delta_1(L)=\langle y_i\,|\,  i\in\N\rangle$;
\item $\SSS(L)$ is solvable of length 3;
\item $\SSS(L)$ is not strongly solvable.
\end{enumerate}
\end{Lemma}
\begin{proof}
The first claim is checked directly.

Denote $H=\langle y_i\,|\,  i\in\N\rangle$.
The basis elements of $\SSS(L)$ are of the first type $y_{i_1}\cdots y_{i_k}\in \SSS(H)$
and the second type $x y_{i_1}\cdots y_{i_k}$,
where $i_1<\cdots <i_k$, and  $k\ge 0$ will be referred to as the length.
Now, $\SSS(H)$ is abelian and the remaining commutators of the basis monomials are as follows:
\begin{align}
\{x y_{i_1}\cdots y_{i_m},y_{j_1}\cdots y_{j_k}\}&=k  y_{i_1}\cdots y_{i_m}y_{j_1}\cdots y_{j_k};\nonumber\\
\{x y_{i_1}\cdots y_{i_m}, x y_{j_1}\cdots y_{j_n}\}&=(n-m) x y_{i_1}\cdots y_{i_m}y_{j_1}\cdots y_{j_n}.
\label{monom2}
\end{align}
Commutators~\eqref{monom2} are non-zero only when $m,n$ have different parities.
Thus, $\delta_1(\SSS(L))=\{\SSS(L),\SSS(L)\}$
contains monomials of the first type and of the second type of odd length.
Similarly, $\delta_2(\SSS(L))$ contains monomials of the first type only.
Finally, we get $\delta_3(\SSS(L))=0$.

To prove the last claim, one checks by induction on $k$ that
$\tilde\delta_k(\SSS(L))$, $k\ge 0$, contain infinitely many monomials
of the second type for both, even and odd, lengths.
Indeed, use ~\eqref{monom2} and the fact that we can additionally multiply the commutators by $y_s$
to make the resulting length even.
\end{proof}

\begin{Lemma}\label{L22b}
Let
$L=\langle x, y_i,z_i \,|\,  [x,y_i]=z_i,\, i\in\N\rangle_K$, $\ch K=2$,
the remaining commutators being trivial.
Then
\begin{enumerate}
\item $\Delta(L)=\Delta_1(L)=\langle y_i, z_i\,|\,  i\in\N\rangle$
      and $L^2=\langle z_i\,|\,  i\in\N\rangle$;
\item $\SSS(L)$ is solvable of length 3;
\item $\SSS(L)$ is not strongly solvable.
\end{enumerate}
\end{Lemma}
\begin{proof}
The first claim is checked directly.

Put $H=\langle y_i,z_i\,|\,  i\in\N\rangle$.
For any $a\in \SSS(H)$ write $a'=\{x,a\}$.
Consider a basis monomial $w=v_1\cdots v_n\in \SSS(H)$, where
$v_i$ denote elements from the basis of $H$. Then
\begin{align}\label{vvv}
w''=(v_1\cdots v_n)''
=\sum_{i=1}^n v_1\cdots v_i''\cdots v_n+2\sum_{1\le i<j\le n} v_1\cdots v_i'\cdots v_j'\cdots v_n=0.
\end{align}
A basis monomial of $\SSS(L)$ without $x$ is of the first type,
otherwise of the second type.
Let $xa,xb,xc,xd\in \SSS(L)$ be all of the second type, where $a,b,c,d\in \SSS(H)$.
Using~\eqref{vvv},
\begin{align*}
\delta_2(xa,xb,xc,xd)&=\{ \{xa,xb\}, \{xc,xd\}\}
                      =\{x(a'b+ab'),x(c'd+cd') \}\\
&=x (a''b+2a'b'+ ab'')(c'd+cd')+x (a'b+ab')(c''d+2c'd'+cd'')=0.
\end{align*}
Let $w_1,w_2,w_3,w_4\in \SSS(L)$ be basis monomials where at least one is of the first type.
One checks that $\delta_2(w_1,w_2,w_3,w_4)$ can be of the first type only.
We proved that $\delta_2(\SSS(L))\subseteq \SSS(H)$.
Finally,  $\delta_3(\SSS(L))=0$.

Let us prove the last claim.
The following observation simplifies our computations below
$$
(y_1z_2+z_1y_2)'=y_1'z_2+y_1z_2'+z_1'y_2+z_1 y_2'=2z_1z_2=0.
$$
We use possibility to multiply additionally in the identity of the strong solvability
\begin{align*}
\{xy_1,xy_2\}&=x(y_1z_2+z_1y_2);\\
\Big\{ \{xy_1,xy_2\}y_5  , \{xy_3,xy_4\}y_6\Big\}
  &=\{x(y_1z_2+z_1y_2)y_5, x(y_3z_4+z_3y_4)y_6\}\\
  &=x (y_1z_2+z_1y_2)(y_3z_4+z_3y_4)(y_5z_6+z_5y_6).
\end{align*}
We continue computations, each time
inserting new variables, like $y_5,y_{6}$  above,  and conclude that
$\tilde \delta_n(\SSS(L))\ne 0$ for all $n\ge 0$.
Thus, $\SSS(L)$ is not strongly solvable.
\end{proof}
\begin{proof}{\it of Theorem~\ref{Tsolv} in case $p=2$.}
It remains to prove that the strong solvability implies that $\dim L^2<\infty$.
Recall that Lemma~\ref{ad} and its Corollary~\ref{Ceigen} remain valid in case $p=2$.
They reduced our proof above to the fact that the Lie algebras of
Proposition~\ref{SpSp} and Proposition~\ref{SpSp2} are not allowed.
The arguments further lead to
Lie algebras of Lemma~\ref{L22} and Lemma~\ref{L22b}, which are not strongly solvable in case $\ch K=2$.
\end{proof}

Two examples above are closely related to the following observation.
\begin{Lemma}\label{L23}
Consider the truncated Hamiltonian Poisson algebra $P=\mathbf{h}_2(K)$
(or the Hamiltonian  Poisson algebra $P=\mathbf{H}_2(K)$),  $\ch K=2$.
Then
\begin{enumerate}
\item $P$ is solvable of length 3.
\item $P$ is not strongly solvable.
\end{enumerate}
\end{Lemma}
\begin{proof}
Let $P=\mathbf{h}_2(K)=K[X,Y]/(X^2,Y^2)=\langle 1,x,y,xy\rangle _K$, where $x,y$ denote the images of $X,Y$.
We have
$\delta_1(P)=\{P,P \}=\langle 1,x,y\rangle _K$,
$\delta_2(P)=\langle 1\rangle _K$, and
$\delta_3(P)=0$.
As above, one checks that $P$ is not strongly solvable.

Let $P=\mathbf{H}_2(K)=K[X,Y]$.
The Poisson brackets of monomials $X^nY^m$, $n,m\ge 0$ depend on parities of $n,m$ of multiplicands.
For simplicity, denote by $X^{\bar 0}Y^{\bar 1}$ all monomials $X^\a Y^\b\in K[X,Y]$ such that $\a$ is even and $\b$ odd, etc.
We get non-zero products only in the cases:
\begin{align*}
\{X^{\bar 1}Y^{\bar 0},X^{\bar 0}Y^{\bar 1}\} &= X^{\bar 0}Y^{\bar 0};\\
\{X^{\bar 1}Y^{\bar 1},X^{\bar 1}Y^{\bar 0}\} &= X^{\bar 1}Y^{\bar 0};\\
\{X^{\bar 1}Y^{\bar 1},X^{\bar 0}Y^{\bar 1}\} &= X^{\bar 0}Y^{\bar 1}.
\end{align*}
Thus, $\delta_1(P)$ is spanned by monomials of three types obtained above.
Consider their commutators, the first line yields that $\delta_2(P)$ is spanned by
monomials of type $Y^{\bar 0}Y^{\bar 0}$. Finally, $\delta_3(P)=0$.
\end{proof}
Thus, the Poisson algebras $\mathbf{h}_2(K)$,
$\mathbf{H}_2(K)$ in characteristic 2 behave similarly to the associative algebra $\mathrm{M}_2(K)$ of $2\times 2$ matrices
in case $\ch K=2$.
\section{Nilpotency and solvability of symmetric algebras $S(L)$}

Finally, in this section we prove Theorem~\ref{Pe16}
that generalizes the result of Shestakov (Theorem~\ref{TSh93}).

\begin{proof} {\it of Theorem~\ref{Pe16}}. Implications
1) $\Rightarrow$ 2) $\Rightarrow$ 3) $\Rightarrow$ 5)
and 1) $\Rightarrow$ 4) $\Rightarrow$ 5) are trivial.

It is sufficient to prove one nontrivial implication 5) $\Rightarrow$ 1).
(In case $p=2$ we prove 4) $\Rightarrow$ 1)).
Consider the case $\ch K\ne 0$, then we can take
the truncated symmetric algebra $\SSS(L)=S(L)/(v^p\,|\,v\in L)$, which is solvable (strongly solvable for $p=2$).
Applying Theorem~\ref{Tsolv}, we get $\dim L^2<\infty$.
By way of contradiction suppose that $L$ is not abelian.
Then there exist $x,y\in L$ such that $[x,y]\ne 0$.
We get a non-abelian finite dimensional subalgebra $\tilde L=\langle x,y\rangle_K+L^2\subseteq L$.

Consider the case $\ch K=0$.
By Theorem~\ref{TSymm0}, $L$ has an abelian ideal $A$ of finite codimension.
Take $x\in L\setminus A$, suppose that there exists $v\in A$ such that $[x,v]\ne 0$.
Consider a subalgebra $\tilde L=\langle x\rangle\oplus \langle (\ad x)^nv\,|\, n\ge 0\rangle \subseteq L$.
By Lemma~\ref{ad}, the action of $\ad x$ is algebraic, thus
$\tilde L$ is a finite dimensional non-abelian subalgebra.
The remaining case is that for any $x\in L\setminus A$ we have $[x,A]=0$.
Now one has $\dim L^2\le \dim L/A+(\dim L/A)^2<\infty$ and the arguments above again yield
a non-abelian finite dimensional subalgebra.

Thus, without loss of generality, we assume that $L$ is a finite dimensional non-abelian Lie algebra.
Since our identities are multilinear,
we also can consider that $K$ is algebraically closed.
It is well known that a finite dimensional non-abelian Lie algebra over an algebraically closed field either
contains the two-dimensional solvable Lie algebra or the three-dimensional nilpotent
Lie algebra (see~\cite{JacLie}, \cite[6.7.1]{Ba}).

Suppose that $L$ contains the two-dimensional solvable Lie algebra $H=\langle x,y\, |\, [x,y]=y\rangle_K$.
Introduce sequences of elements:
$e_{k}=xy^{2^k-1}$, $h_{k}=xy^{2^k}$, $f_{k}=xy^{2^k+1}$ for all $k\ge 1$.
Consider the case $\ch K\ne 2$.
We check by induction on $k\ge 1$ that $e_k,h_k,f_k\in \delta_{k-1}(S(H))$.
The base of induction $k=1$ is trivial.
Let $k>1$, the inductive step follows by relations
\begin{align*}
\delta_{k}(S(H))\ni \{e_{k}, h_{k}\} &=\{xy^{2^k-1},xy^{2^k}\}=x y^{2^{k+1}-1}=e_{k+1};\\
\delta_{k}(S(H))\ni\{h_{k}, f_{k}\} &=\{xy^{2^k}  ,xy^{2^k+1}\}=x y^{2^{k+1}+1}=f_{k+1};\\
\delta_{k}(S(H))\ni\{e_{k}, f_{k}\} &=\{xy^{2^k-1},xy^{2^k+1}\}=2 x y^{2^{k+1}}=2h_{k+1}.
\end{align*}
Thus, $\delta_k(S(H))\ne 0$ for all $k\ge 1$, a contradiction with solvability of $S(L)$.
In case $\ch K=2$, we check by induction that
$e_k,h_k,f_k\in \tilde\delta_{k-1}(S(H))$, $k\ge 1$, using relations above and
$$
\tilde \delta_{k}(S(H))\ni\{e_{k}, h_{k}\} y=\{xy^{2^k-1},xy^{2^k}\}y =x y^{2^{k+1}}=h_{k+1}.
$$
Thus, $S(H)$ is not strongly solvable, a contradiction.

Suppose that $L$ contains
the three-dimensional nilpotent Lie algebra $H=\langle x,y,z\, |\, [x,y]=z,[x,z]=[y,z]=0\rangle_K$.
Introduce sequences of elements:
$e_{k}=xy^{2^k}z^{2^k-1}$, $h_{k}=xy^{2^k+1}z^{2^k-1}$, $f_{k}=xy^{2^k+2}z^{2^k-1}$ for all $k\ge 1$.
Consider the case $\ch K\ne 2$.
We check by induction on $k\ge 1$ that $e_k,h_k,f_k\in \delta_{k-1}(S(H))$.
The base of induction $k=1$ is trivial.
Let $k>1$, the inductive step follows by relations
\begin{align*}
\delta_{k}(S(H))\ni \{e_{k}, h_{k}\} &=\{xy^{2^k}z^{2^k-1}, xy^{2^k+1}z^{2^k-1}\}=x y^{2^{k+1}}z^{2^{k+1}-1}=e_{k+1};\\
\delta_{k}(S(H))\ni\{h_{k}, f_{k}\} &=\{xy^{2^k+1}z^{2^k-1}  ,xy^{2^k+2}z^{2^k-1}\}=x y^{2^{k+1}+2}z^{2^{k+1}-1}=f_{k+1};\\
\delta_{k}(S(H))\ni\{e_{k}, f_{k}\} &=\{xy^{2^k}z^{2^k-1}, xy^{2^k+2}z^{2^k-1}\}=2 x y^{2^{k+1}+1}z^{2^{k+1}-1}=2h_{k+1}.
\end{align*}
In case $\ch K=2$, we check by induction that
$e_k,h_k,f_k\in \tilde\delta_{k-1}(S(H))$, $k\ge 1$, using relations above and
$$
\tilde\delta_{k}(S(H))\ni \{e_{k}, h_{k}\}y =\{xy^{2^k}z^{2^k-1}, xy^{2^k+1}z^{2^k-1}\}y=x y^{2^{k+1}+1}z^{2^{k+1}-1}=h_{k+1}.
$$
Thus, $S(H)$ is not strongly solvable, a contradiction, which proves that $L$ is abelian.
\end{proof}

The question of the solvability of the symmetric algebra  $S(L)$ in case $\ch K=2$
is more complicated as shown below.
The same Lie algebras of Lemma~\ref{L22} and Lemma~\ref{L22b},
also yield solvable symmetric algebras, of course, they are not strongly solvable by that Lemmas.

\begin{Lemma}\label{LS2solvA}
Let $L=\langle x,y_i\,|\,[x,y_i]=y_i,i\in\N\rangle_K$,
other commutators being trivial, $\ch K=2$.
The symmetric Poisson algebra $S(L)$ is solvable of length 3 and not strongly solvable.
\end{Lemma}
\begin{proof}
Denote $H=\langle y_i\,|\,i\in\N\rangle_K$.
For a monomial $v=y_{i_1}y_{i_2}\cdots y_{i_k}\in S(H)$ introduce its length $|v|=k$.
Then $v'=\{x,v\}=|v|v$.
A basis of $S(L)$ is formed by
$x^\a v$, $\a\ge 0$, where $v\in S(H)$ are respective basis monomials.
Consider products:
\begin{equation} \label{xaaxbb2}
\{ x^\a v,x^\b w\} =x^{\a+\b-1}(\a|w|+ \beta|v|) vw.
\end{equation}
These products depend on parities of $\a,\b,|v|,|w|$.
For simplicity, denote by $x^{\bar 0}v^{\bar 1}$ all monomials $x^\a v\in S(L)$ such that $\a$ is even and $|v|$ is odd, etc.
The only non-zero products~\eqref{xaaxbb2} are of types:
\begin{align*}
\{x^{\bar 1}v^{\bar 0},x^{\bar 0}v^{\bar 1}\} &= x^{\bar 0}v^{\bar 1};\\
\{x^{\bar 1}v^{\bar 1},x^{\bar 1}v^{\bar 0}\} &= x^{\bar 1}v^{\bar 1};\\
\{x^{\bar 1}v^{\bar 1},x^{\bar 0}v^{\bar 1}\} &= x^{\bar 0}v^{\bar 0}.
\end{align*}
Thus, $\delta_1(S(L))$ is spanned by monomials of three types obtained above.
Consider their commutators, the last line yields that $\delta_2(S(L))$ is spanned by
monomials of type $x^{\bar 0}v^{\bar 0}$. Finally, $\delta_3(S(L))=0$.
\end{proof}

\begin{Lemma}\label{LS2solvB}
Let $L=\langle x,y_i,z_i\,|\,[x,y_i]=z_i,i\in\N\rangle_K$, other commutators being trivial, $\ch K=2$.
The symmetric Poisson algebra $S(L)$ is solvable of length 3 and not strongly solvable.
\end{Lemma}
\begin{proof}
Set $H=\langle y_i,z_i\,|\,i\in\N\rangle_K$.
As above, denote $a'=\{x,a\}$ for $a\in S(H)$.
Consider basis elements $x^\a a,x^\b b,x^\gamma c,x^\delta d\in S(L)$, where
$a,b,c,d\in S(H)$ are basis monomials and $\a,\b,\gamma,\delta\ge 0$. Then
\begin{align}
\label{xaaxbb}
\{ x^\a a,x^\b b \}        &=x^{\a+\b-1}(\a ab'+\beta  a'b);\\
\{ x^\gamma c,x^\delta d\} &=x^{\gamma+\delta-1}(\gamma cd'+\delta  c'd).\nonumber
\end{align}
We use these computations and the observation above~\eqref{vvv}, that $a''=0$ for any $a\in S(H)$.
\begin{align*}
&\delta_2(x^\a a,x^\b b, x^\gamma c,x^\delta d)
=\Big\{x^{\a+\b-1}(\a ab'+\beta  a'b), x^{\gamma+\delta-1}(\gamma cd'+\delta  c'd) \Big \}\\
&=x^{\a+\b+\gamma+\delta-3}\Big((\a+\b-1)(\gamma+\delta)(\a ab'+\beta  a'b)c'd'
+(\gamma+\delta-1)(\a+\b)a'b'(\gamma cd'+\delta  c'd)\Big).
\end{align*}
The first summand is nonzero only in the case $\a+\b$ is even and $\gamma+\delta$ odd,
while the second one is non-zero in the opposite case.
Both cases yield that $\a+\b+\gamma+\delta-3$ is even.
Thus, $\delta_2(S(L))$ is spanned by monomials of type $x^\a a $, $a\in S(H)$, and $\a$ even.
Finally, by~\eqref{xaaxbb}, $\delta_3(S(L))=0$.
\end{proof}

\subsection*{Acknowledgments}
The authors are grateful to Raimundo Bastos, Alexei Krasilnikov,  Plamen Koshlukov, and Ivan Shestakov for useful discussions.
The authors are especially grateful to Alexei Krasilnikov for explaining
the mystery of products of commutators in associative algebras.

\end{document}